\documentclass[11pt,hidelinks]{article}
\usepackage{orcidlink}

\title{Gaussian deconvolution and the lace expansion}

\author{
Yucheng Liu\,\orcidlink{0000-0002-1917-8330}\thanks{Department of Mathematics,
	University of British Columbia,
	Vancouver, BC, Canada V6T 1Z2.
	Liu: \href{mailto:yliu135@math.ubc.ca}{yliu135@math.ubc.ca}.
	Slade: \href{mailto:slade@math.ubc.ca}{slade@math.ubc.ca}.
	}
\and
Gordon Slade\,\orcidlink{0000-0001-9389-9497}$^*$
}
\date{July 26, 2024}  

\usepackage{amsmath, amssymb, amscd, amsthm, amsfonts}
\usepackage{graphicx} 
\usepackage{hyperref} 
\usepackage{booktabs}
\usepackage{xcolor}
\usepackage{ dsfont, bm}
\usepackage[title]{appendix}
\usepackage{comment}
\usepackage{enumerate}
\usepackage{enumitem}
\usepackage{cite}

\usepackage[textwidth=480pt,textheight=650pt,centering]{geometry} 

\theoremstyle{plain}
\newtheorem{theorem}{Theorem}[section]
\newtheorem{lemma}[theorem]{Lemma}

\newtheorem{proposition}[theorem]{Proposition}
\newtheorem{corollary}[theorem]{Corollary}

\newtheorem{definition}[theorem]{Definition}
\newtheorem{assumption}[theorem]{Assumption}
\newtheorem{remark}[theorem]{Remark}

\numberwithin{equation}{section}

\newcommand{\ie}{i.e.}
\newcommand{\eg}{e.g.}

\newcommand{\eps}{\varepsilon}

\newcommand{\Z}{\mathbb{Z}}

\newcommand{\R}{\mathbb{R}}
\newcommand{\C}{\mathbb{C}}

\newcommand{\T}{\mathbb{T}}

\newcommand{\Fcal}{\mathcal{F}}

\newcommand{\grad}{\nabla}
\newcommand{\inv}{^{-1}}
\renewcommand{\(}{\left(}
\renewcommand{\)}{\right)}
\newcommand{\half}{\frac{1}{2}}

\newcommand{\1}{\mathds{1}}

\newcommand{\nl}{\nonumber \\}

\providecommand{\abs}[1]{\lvert#1\rvert}
\providecommand{\norm}[1]{\lVert#1\rVert}
\providecommand{\Bigabs}[1]{\Big\lvert#1\Big\rvert}
\providecommand{\Bignorm}[1]{\Big\lVert#1\Big\rVert}
\providecommand{\bignorm}[1]{\big\lVert#1\big\rVert}

\usepackage{mathrsfs}

\newcommand{\uvec}{\tilde{u}}

\newcommand{\veee}[1]{|\!|\!|#1|\!|\!|}
\newcommand{\xvee}{\veee{x}}
\providecommand{\nnnorm}[1]{\veee #1}

\providecommand{\floor}[1]{\lfloor #1 \rfloor}

\newcommand{\Dnn}{D}
\newcommand{\Fpp}{\kappa_0}

\newcommand{\crit}{_{z_c}}

\begin{document}
\maketitle

\begin{abstract}
We give conditions on a real-valued function $F$ on $\Z^d$, for $d>2$, which ensure that
the solution $G$ to the convolution equation $(F*G)(x) = \delta_{0,x}$ has Gaussian
decay $|x|^{-(d-2)}$ for large $|x|$.
Precursors of our results were obtained in the 2000s, using intricate Fourier analysis.
In 2022, a very simple deconvolution theorem was proved, but its applicability was
limited.   We extend the 2022 theorem to remove its limitations
while maintaining its simplicity---our main tools are
H\"older's inequality, weak derivatives, and basic Fourier theory in $L^p$ space.
Our motivation comes from critical phenomena
in equilibrium statistical mechanics, where the convolution equation is provided
by the lace expansion and
$G$ is a critical two-point function.
Our results significantly
simplify existing proofs of critical $|x|^{-(d-2)}$
decay in high dimensions for self-avoiding walk, Ising and $\varphi^4$
models, percolation, and
lattice trees and lattice animals.
We also improve previous error estimates.
\end{abstract}

%


\centerline{
\it Dedicated with appreciation to Geoffrey Grimmett
}

\section{Gaussian deconvolution}
\label{sec:intro}
\subsection{Introduction}

Since its introduction by Brydges and Spencer in 1984, the lace expansion
has been extended and applied to prove detailed results concerning mean-field
critical behaviour of statistical mechanical models on $\Z^d$ above their upper critical
dimensions.  These results include self-avoiding walk and spin systems
(Ising and $\varphi^4$) in dimensions $d>4$ (e.g., \cite{BS85,BHH21,HS92a,Saka07,Saka15}),
percolation for $d>6$ \cite{HS90a,HH17book,FH17}, and lattice trees and lattice animals (LTLA)
for $d>8$ \cite{HS90b,FH21}.
An aspect of this work that is
central to many applications
is a proof that the critical two-point
function for each of the mentioned models has Gaussian decay $|x|^{-(d-2)}$ at large distance.
A proof of this $x$-space decay was not discovered until suitable deconvolution
theorems were proved in the 2000s \cite{Hara08,HHS03}.
This followed two decades of work
in which the Fourier transform of the critical two-point function was
proved to behave as $|k|^{-2}$ for small frequencies $k$.  Knowledge of $|k|^{-2}$
behaviour is sufficient for much progress, including proving the triangle condition \cite{AN84}
for percolation.  However,
\cite[Example~1.6.2]{MS93} shows that
it does not imply $|x|^{-(d-2)}$ decay.  The $|x|^{-(d-2)}$ decay has been
employed, e.g., in percolation theory to construct and analyse the incipient infinite
cluster \cite{HJ04,KN09},
to study arm exponents \cite{CHS23,KN11}, and to study finite-size scaling \cite{Aize97,CH20,HMS23}
where in particular the near-critical decay $|x|^{-(d-2)}e^{-c(p_c-p)^{1/2}|x|}$
of the two-point function has been
established and exploited in \cite{HMS23}.

In 2018, a simpler proof of an $|x|^{-(d-2)}$ upper bound for the weakly self-avoiding
walk in dimensions $d>4$ was presented in \cite{BHK18}, using Banach algebras
and convolution estimates and without using Fourier transforms.  A
different and strikingly simple
proof of $|x|^{-(d-2)}$ decay for weakly self-avoiding
walk in dimensions $d>4$ appeared very recently in \cite{Slad22_lace},
where an asymptotic formula was obtained by using the Fourier transform
in an elementary way
(inspired by Kotani's Theorem \cite{Slad20_Kotani}).
However, in \cite{Slad22_lace} obstacles were identified
to extending the proof to other models including percolation and
LTLA.
In this paper, we remove the obstacles while maintaining the simplicity
of the approach of \cite{Slad22_lace}, and we extend it
so that its realm of applications includes all models for which $|x|^{-(d-2)}$ decay has been proved
via the lace expansion, namely self-avoiding walk,  Ising and $\varphi^4$ spin models,
percolation, LTLA.

This extension provides a new proof of Hara's Gaussian Lemma \cite{Hara08}.
Our method involves an elementary but judicious application of
H\"older's inequality in the context of Fourier analysis for weak derivatives in $L^p$ spaces
(we recall the necessary theory in Appendix~\ref{appendix:weak}).
We also obtain improved error bounds compared to \cite{Hara08},
using a fractional derivative analysis.

More generally, our work provides conditions on a function
$F$ on $\Z^d$ so that the solution $G$ to the convolution equation
$(F*G)(x)=\delta_{0,x}$
has $|x|^{-(d-2)}$ decay at large distance.
In applications, the verification of the assumed
conditions uses the model-dependent diagrammatic estimates common in lace expansion analysis;
we do not discuss that further as the details can be found in \cite{Hara08,HHS03,Saka07,BHH21}
and we have nothing new to say in that regard.

In \cite{Slad22_lace}, it was also noted that further work would be needed in order to apply the method to spread-out models, which are used to analyse dimensions only slightly above the upper critical dimension.
In a companion paper \cite{LS24b}, we adapt our
deconvolution theorem to spread-out models
and use it to significantly simplify the proof, compared
to \cite{HHS03}, of
$|x|^{-(d-2)}$ decay of critical two-point functions for
spread-out versions of self-avoiding walk, Ising model, percolation, and LTLA above their upper critical dimensions.

\medskip\noindent
{\bf Notation:}  We write
$a \vee b = \max \{ a , b \}$ and $a \wedge b = \min \{ a , b \}$.
We write
$f = O(g)$ or $f \lesssim g$ to mean there exists a constant $C> 0$ such that $\abs {f(x)} \le C\abs {g(x)}$, and $f= o(g)$ for $\lim f/g = 0$.
To avoid dividing by zero, with $|x|$ the Euclidean norm of $x\in \R^d$, we define
\begin{equation}
\label{eq:xvee1}
	\xvee = \max\{|x| , 1\}.
\end{equation}
Note that \eqref{eq:xvee1} does not define a norm on $\R^d$.

\subsection{Gaussian deconvolution theorem}

The convolution of two absolutely summable functions $f,g:\Z^d \to \R$ is
the function $f*g:\Z^d \to \R$ defined by
\begin{equation}
    (f*g)(x) = \sum_{y\in\Z^d} f(y)g(x-y).
\end{equation}
Let $d>2$,
let $\Dnn: \Z^d \to [0,1)$ be
defined by $\Dnn(x) = \frac{1}{2d} \1_{|x|=1}$,
and let
$\mu \in (0,1]$.
The lattice Green function $C_\mu : \Z^d \to \R$
is the solution, vanishing
at infinity, of
\begin{equation}
\label{eq:LGF}
    (\delta -\mu \Dnn)*C_\mu = \delta ,
\end{equation}
where $\delta$ is the Kronecker delta $\delta(x)=\delta_{0,x}$.
In particular,
the case $\mu=1$ involves $\delta-\Dnn$, which is minus the
discrete Laplacian, and $C_1(x)$ is the expected number of visits to $x$ by a
simple random walk started from the origin.
The solution to \eqref{eq:LGF} can be obtained using the (inverse) Fourier transform
\begin{equation}
\hat f(k)  = \sum_{x\in\Z^d}f(x) e^{ik\cdot x} 	\quad (k \in \T^d),
\qquad
f(x) = \int_{\T^d} \hat f(k) e^{-ik\cdot x}  \frac{ dk }{ (2\pi)^d } \quad (x\in \Z^d),
\end{equation}
where $\T^d=(\R/2\pi\Z)^d$ is the continuum torus.
We identify $\T^d$ with $(-\pi,\pi]^d \subset \R^d$.
In particular, in dimensions $d>2$,
$C_\mu$ is given by the absolutely convergent integral
\begin{equation}
\label{eq:Cmux}
    C_\mu(x) =
    \int_{\T^d} \frac{ e^{-ik\cdot x} }{1 - \mu\hat \Dnn(k)} \frac{ dk }{ (2\pi)^d },
    \qquad
    \hat \Dnn(k) = d^{-1}\sum_{j=1}^d \cos k_j.
\end{equation}
When $\mu\in (0,1)$, $C_\mu$ decays exponentially (detailed asymptotics
are given in \cite{MS22}), whereas for $\mu=1$
the large-$x$ decay of $C_1(x)$ is well-known
(e.g., \cite{LL10})
to be given by
\begin{equation}
\label{eq:C1_asymp}
    C_1(x) =
    \frac { a_d }{\xvee^{d-2} }
    + O\(\frac 1 {\xvee^d }\) ,
    \qquad a_d = \frac{ d \Gamma(\frac{ d-2 } 2 ) }{ 2\pi^{d/2}}.
\end{equation}
We do not give a new proof of \eqref{eq:C1_asymp}.
Although $C_1$ is not summable,
it is nevertheless a standard fact that
\eqref{eq:Cmux} does give a solution to \eqref{eq:LGF} when $\mu=1$.

Motivated by applications to statistical mechanics based on the lace expansion,
given a function $F:\Z^d \to \R$ we wish to understand
the large-$x$ decay of the solution $G$ to the equation
\begin{equation}
\label{eq:FG}
    F*G =\delta,
\end{equation}
given by the Fourier integral
\begin{equation} \label{eq:Gint}
    G(x)= \int_{\T^d}\frac{e^{-ik\cdot x}}{\hat F(k)} \frac{dk}{(2\pi)^d}.
\end{equation}
We are interested in functions $F$ satisfying the following assumptions,
which in particular imply the absolute convergence of the above integral.

\begin{assumption}
\label{ass:F}
We assume that $F$ is a $\Z^d$-symmetric function (invariant under
reflection in coordinate hyperplanes
and rotation by $\pi/2$) for which
there are
$K_1, K_2 > 0$ and $\rho > \frac{d-8}{2} \vee 0$ such that, for all $x \in \Z^d$ and $k \in \T^d$,
\begin{align}
\label{eq:Fass}
\abs{ F(x) } &\le \frac{ K_1 }{ \xvee^{d+2+\rho}},
	\qquad 	\hat F(0) \ge 0, 	
	\qquad 	\hat F(k)-\hat F(0) \ge K_2 \abs k^2  .
\end{align}
\end{assumption}

In practice, the verification of Assumption~\ref{ass:F} requires a
small parameter, which arises from a weak interaction or a spread-out model or
a high-dimension condition.  Once we assume \eqref{eq:Fass}, there is no
small parameter appearing explicitly in our analysis.
In \cite{LS24b}, we indicate the role of the small parameter in detail, for
spread-out models.

By Taylor's Theorem and Assumption~\ref{ass:F},
\begin{equation}
\label{eq:K0def}
    \Fpp = -\sum_{x\in \Z^d} |x|^2F(x) \in
    [2dK_2,\infty),  
\end{equation}
where the equality defines the constant $\Fpp$.
Assumption~\ref{ass:F} is satisfied by
$F=\delta -\mu \Dnn$ for $\mu\in (0,1]$ for any $\rho>0$.
In particular, the \emph{infrared bound} (lower bound on $\hat F(k)-\hat F(0)$) in this case
is simply
\begin{align}
\label{eq:IR-SRW}
\mu \hat \Dnn(0) -  \mu\hat \Dnn(k)
= \frac \mu d \sum_{j=1}^d (1 - \cos k_j) \ge \frac {2\mu} {\pi^2 d}\abs k^2 \qquad (k\in \T^d).
\end{align}
We do not assume that $F(x) \le 0$ for all $x \neq 0$, unlike
analogues of our results such as \cite{Uchi98}.
This is important for applications to the lace expansion, where positive
values of $F$ do occur.

Let
\begin{align}
\label{eq:lambda_z}
\lambda &= \frac{1} { \hat F(0) - \sum_{x\in\Z^d} \abs x^2 F(x) }
	= \frac 1 { \hat F(0) + \Fpp},
\qquad \mu = 1 - \lambda \hat F(0).
\end{align}
By hypothesis, $\lambda >0$, so $\mu \le 1$.  Also, $\lambda\hat F(0) < 1$
by definition of $\lambda$, so in
fact $\mu\in (0,1]$.
For the \emph{critical} case of $\hat F (0)=0$, we have $\mu =1$ and
$\lambda = \Fpp \inv$.

The following is our general Gaussian deconvolution theorem.
Its proof is given in Section~\ref{sec:pf-mr}.

\begin{theorem}[Gaussian deconvolution]
\label{thm:gaussian_lemma}
Let $d >2$, and let
$F$ satisfy
Assumption~\ref{ass:F}.
Then
the solution $G$ to \eqref{eq:FG},
given by the Fourier integral \eqref{eq:Gint},
satisfies
\begin{align}
\label{eq:G_asymp}
    G(x) =
    \lambda C_{\mu}(x)
    +
     O\(\frac 1 {\xvee^{d-2+s}}\)
\end{align}
for any power $s$ which obeys
\begin{alignat}{2}
    &s < \rho \wedge 2   \qquad   && (2<d\le 8), \label{eq:s_small}
    \\
    &s < (\rho - \tfrac{d-8}{2}) \wedge 2   \qquad   && (d>8). \label{eq:s_large}
\end{alignat}
The error estimate in \eqref{eq:G_asymp} depends only on $d,K_1,K_2,\rho,s$.
\end{theorem}

In the critical case, using $\mu = 1$ and the large-$x$ decay of $C_1(x)$ in \eqref{eq:C1_asymp},
the conclusion \eqref{eq:G_asymp} becomes,
with $\Fpp$ given by \eqref{eq:K0def}
and the same restriction on $s$,
\begin{equation} \label{eq:Gcrit_asymp}
    G(x) = \frac { a_d }{\Fpp \xvee^{d-2} }
    +
    O\(\frac 1 {\xvee^{d-2+s}}\)
    .
\end{equation}
When $\mu\in (0,1)$, which occurs in the subcritical case $\hat F(0)>0$,
Theorem~\ref{thm:gaussian_lemma} remains useful in the sense that,
although $C_\mu$ decays exponentially (its detailed asymptotics
are given in \cite{MS22}), we nevertheless obtain an upper bound
$G(x) \le O( \xvee^{-(d-2)})$ which is uniform no matter how close to
zero $\hat F(0)$ is.
We expect that Theorem~\ref{thm:gaussian_lemma} remains true for any $\rho>0$
(without our restriction that $\rho>\frac{d-8}{2}$) and that
the optimal power in the error estimate \eqref{eq:G_asymp} has
any $s< \rho\wedge 2$ for all dimensions $d>2$.

In fact,
Theorem~\ref{thm:gaussian_lemma} was first proved
(in the critical case) by Hara~\cite{Hara08}
using intricate Fourier analysis,
with error exponent\footnote{A ``slightly lengthy''
unpublished improvement
of the error term to the presumably optimal $s=\rho \wedge 2$ is mentioned in \cite[(1.39)]{Hara08}.} $s=(\rho\wedge 2)/d$
and without the assumption that $\rho > \frac{d-8}{2}$.
This assumption, which we need
for the reason indicated in Remark~\ref{rk:rho_restriction},
is satisfied for all known applications (see Section~\ref{sec:app_stat_mec}).
Our proof of Theorem~\ref{thm:gaussian_lemma} is
completely different from the proof in \cite{Hara08} and is much simpler.
For the weaker statement that in the critical case
(with $\rho > \frac{d-8}{2}\vee 0$ as in Assumption~\ref{ass:F})
\begin{equation}
\label{Gx-simple}
    G(x)=\frac { a_d }{\Fpp\xvee^{d-2} } +  o\(\frac 1 {\xvee^{d-2+s}}\)
    \
    \text{as $\abs x \to \infty$ with}
    \
    s = \begin{cases}
    0 &
    (\rho \le 1 + (\frac{d-8}{2}\vee 0))
    \\
    1 & (\rho > 1 + (\frac{d-8}{2}\vee 0)),
    \end{cases}
\end{equation}
which is sufficient
for applications, our proof is remarkably short and simple.
The improved error bound in Theorem~\ref{thm:gaussian_lemma} requires more effort,
but remains simple conceptually.

Applications to the Ising model, percolation and LTLA
involve
an inhomogeneous
convolution
equation of the form
\begin{align}
\label{eq:FHg}
F * H = g.
\end{align}
Suppose that $g$ is $\Z^d$-symmetric and satisfies
\begin{align}
\label{eq:g_decay}
\abs{ g(x) } \le \frac{ K_3 }{ \xvee^{d+ \rho \wedge 2} }
\end{align}
for some $K_3 >0$.
We are interested in solutions to \eqref{eq:FHg} given by the Fourier integral
\begin{align} \label{eq:H-fourier}
H(x) = \int_{\T^d}\frac{ \hat g(k)  e^{-ik\cdot x}}{\hat F(k)} \frac{dk}{(2\pi)^d}.
\end{align}
The following simple observation was noted in Hara \cite{Hara08} (with $s = (\rho\wedge 2)/d))$.

\begin{corollary}
\label{cor:general}
Let $d>2$, and let $F$ satisfy Assumption~\ref{ass:F} with $\hat F(0)=0$ (critical case)
and with $g$ obeying \eqref{eq:g_decay}.
Then the solution $H$ to \eqref{eq:FHg}, given by the Fourier integral
\eqref{eq:H-fourier},
satisfies
\begin{align}
H(x) = \frac{ a_d \sum_y g(y) }{ \Fpp  \xvee^{d-2} } + O\(\frac 1 {\xvee^{d-2+s}}\)
\end{align}
with $s$ as in \eqref{eq:s_small}--\eqref{eq:s_large}.
The constant in the error term depends only on $d, K_1, K_2, K_3, \rho, s$.
\end{corollary}

\begin{proof}
Let $G$ be the solution to $F * G = \delta$ given by Theorem~\ref{thm:gaussian_lemma}.
It suffices to prove that $H = g*G$,
since the desired result then follows from the asymptotic formula for $G$ in \eqref{eq:Gcrit_asymp} and the elementary convolution estimate \cite[Proposition~1.7(ii)]{HHS03}.
It might appear evident from \eqref{eq:FHg} or \eqref{eq:H-fourier} that indeed $H = g * G$, but this is a subtle point because $G$ is not in $\ell^1(\Z^d)$.
We therefore proceed as follows.

Let $\mathcal{F}^{-1}:L^1(\T^d) \to \ell^\infty(\Z^d)$ denote the
inverse Fourier transform.
Given $g$ obeying \eqref{eq:g_decay},
we define two bounded linear maps $T_1$ and $T_2$ from $L^1(\T^d)$
to $\ell^\infty(\Z^d)$ by
\begin{align}
    T_1(\hat\varphi) = \mathcal{F}^{-1}(\hat g\hat\varphi),
    \qquad
    T_2(\hat\varphi) = g* \mathcal{F}^{-1}(\hat\varphi).
\end{align}
The maps are well defined since $\hat g \in L^\infty(\T^d)$
and $g \in \ell^1(\Z^d)$.
If $\hat\varphi$ is $C^\infty$, then $\varphi=\mathcal{F}^{-1}(\hat\varphi) \in \ell^1(\Z^d)$ and
$T_1(\hat\varphi)= g*\varphi=T_2(\hat\varphi)$.  The bounded linear maps
therefore agree on a dense subset of $L^1(\T^d)$, so they agree on all of
$L^1(\T^d)$.  In particular,
\begin{align}
    H= T_1(1/\hat F) = T_2(1/\hat F) = g*G   .
\end{align}
This completes the proof.
\end{proof}

\subsection{Application to statistical mechanical models}
\label{sec:app_stat_mec}

For many statistical mechanical models in the high-dimensional setting, the lace expansion provides convolution equations
for critical and subcritical two-point functions. Our results can be applied to obtain their $x$-space asymptotics.
For nearest-neighbour models, Assumption~\ref{ass:F} has been verified
with
\begin{equation}
\label{eq:rho-models-NN}
    \rho =
    \begin{cases}
    2(d-4) & \text{(self-avoiding walk, $d \ge 5$ \cite{Hara08})}
    \\
    2(d-4) & \text{(Ising, $d$ sufficiently large \cite{Saka07})}
    \\
    2(d-4) & \text{(1- or 2- component weakly-coupled $\varphi^4$, $d \ge 5$ \cite{BHH21})}
    \\
    d-6 & \text{(percolation, $d \ge 11$ \cite{FH17})}
    \\
    d-10 & \text{(LTLA, $d \ge 27$ \cite{FH21})}.
    \end{cases}
\end{equation}
(For the Ising model, percolation and LTLA, \eqref{eq:g_decay} is also verified
in the above references.)
The restriction $\rho > \frac{d-8}{2}\vee 0$ is satisfied in all cases.
See also \cite{Saka15} for $\varphi^4$.
The above references all make use of Hara's Gaussian Lemma
\cite[Section~1.2.2]{Hara08}.  Our results provide a simpler and more conceptual alternative.
In addition, our error bounds in Theorem~\ref{thm:gaussian_lemma}
translate directly into the same error bounds for the critical two-point functions
of all the above models, which improves
previous results.
For example, for the critical two-point function of the
nearest-neighbour strictly self-avoiding walk in dimensions $d \ge 5$,
the use
of Theorem~\ref{thm:gaussian_lemma}
to replace the Gaussian Lemma in
\cite{Hara08}
gives
\begin{equation}
    G_{z_c}(x) = \frac{a_d}{\Fpp\xvee^{d-2}} + O\( \frac{1}{\xvee^{d-\eps}} \)
\end{equation}
for any $\eps >0$.
Similarly, for the critical two-point function of nearest-neighbour percolation in dimensions $d \ge 11$, combined with \cite{FH17,Hara08}, Corollary~\ref{cor:general} gives
\begin{equation}
    H\crit(x) = \frac{ A\crit }{\xvee^{d-2}} + O\( \frac{1}{\xvee^{d-\eps}} \)
\end{equation}
for any $\eps >0$,
with an explicit amplitude $A\crit$.
The errors computed in \cite{Hara08}  instead involve
$\eps=2-2/d$.

\section{Proof of Gaussian deconvolution Theorem~\ref{thm:gaussian_lemma}}
\label{sec:pf-mr}

\subsection{Outline of proof}

The structure of the proof of Theorem~\ref{thm:gaussian_lemma} is very simple.

\medskip\noindent
{\bf Isolation of leading term.}
Suppose $F$ satisfies Assumption~\ref{ass:F}.
For $\mu\in(0,1]$, let $A_\mu = \delta - \mu \Dnn$,
so that $A_\mu * C_\mu = \delta$ by \eqref{eq:LGF}.
For any $\lambda\in \R$, we write
\begin{align}
\label{eq:GC}
G
&= \lambda C_\mu + \delta * G - \lambda C_\mu * \delta	\nl
&= \lambda C_\mu + (C_\mu * A_\mu) * G - \lambda C_\mu * (F * G)	\nl
&= \lambda C_\mu + C_\mu * E_{\lambda,\mu} * G,
\end{align}
with $E_{\lambda,\mu} = A_\mu - \lambda F$.
This isolates the leading term $\lambda C_\mu$:
\begin{align}
 \label{eq:G_isolate-intro}
G
&= \lambda C_\mu + f,
\qquad
f = f_{\lambda,\mu} =  C_\mu * E_{\lambda,\mu} * G.
\end{align}
The first moment of $E$ vanishes by our symmetry assumptions,
and our specific choice of $\lambda$ and $\mu$ in \eqref{eq:lambda_z}
has been made to ensure that
\begin{align}
\label{eq:E_condition}
\sum_{x\in\Z^d} E_{\lambda,\mu}(x) = \sum_{x\in\Z^d} \abs x^2 E_{\lambda,\mu}(x) = 0.	
\end{align}
Indeed, \eqref{eq:E_condition} is a system of two linear equations in $\lambda,\mu$ whose solution is given by \eqref{eq:lambda_z}.

\medskip\noindent
{\bf Error term with integer power.}
Consider Theorem~\ref{thm:gaussian_lemma}
in the critical case, where the Fourier transforms
$\hat G(k)$ and $\hat C(k)$ both behave like $|k|^{-2}$.
As we will show, the vanishing moments of
$E$ in \eqref{eq:E_condition} imply that its Fourier transform is bounded by $|k|^{2+\sigma}$
for any $0<\sigma <\rho\wedge 2$, so that
$\hat f(k) = \hat C(k)\hat E(k) \hat G(k)$ is of order $|k|^{\sigma-2}$, which
is less singular than the main term $\hat C(k)$.
Roughly speaking, this suggests that a derivative of $\hat f$ of order $a$ will have
singularity $|k|^{\sigma-2-a}$, which is integrable provided that
$a< d-2+\sigma$.  For example, if $\sigma\in (1,2)$ then $a=d-1$ is permitted.

Given a multi-index $\alpha=(\alpha_1,\ldots,\alpha_d)$ with each $\alpha_i$
a nonnegative integer, we write $|\alpha|=\sum_{i=1}^d\alpha_i$
and define the differential operator $\nabla^\alpha
= \frac{\partial^{|\alpha|}}{\partial k_1^{\alpha_1}\cdots
\partial k_d^{\alpha_d}}$.
As we explain later in more detail,
an elementary result from Fourier analysis states
that if $a$ is an integer and
the derivatives $\nabla^\alpha \hat f(k)$ are in $L^1(\T^d)$ for
$|\alpha| \le a$, then $f(x)$ decays at least as $o(|x|^{-a})$.  This allows for
a proof of Theorem~\ref{thm:gaussian_lemma} with integer-power error estimate, by taking $a=d-2$
or $a=d-1$ derivatives of $\hat f$.

An important consideration is that the calculation of $\grad^\alpha \hat f$ involves calculating $\grad^\alpha \hat F$.
If we are to take $\abs \alpha = d-2$ derivatives, since the sum $\sum_{x\in\Z^d} \abs x^{d-2} / \abs x^{d+2+\rho}$ diverges when $d - 2 \ge 2 + \rho$,
the function $|x|^{d-2} F(x)$
may not be in $\ell^1(\Z^d)$ if $\rho \le d-4$
and hence
we are not able to conclude that $\grad^\alpha \hat F$ exists.
This potentially dangerous situation arises
for percolation and LTLA, and was pointed out in \cite{Slad22_lace} as an obstacle.
We overcome this obstacle using weak derivatives, which allows us to make sense of $\grad^\alpha \hat F$ for $\abs \alpha < \half d + 2 + \rho$;
the case $|\alpha|=d-2$ gives rise to our restriction $\rho > \frac{d-8}{2} $
(see Remark~\ref{rk:rho_restriction}).

\medskip\noindent
{\bf Improved error term.}
In Section~\ref{section:proof_frac} we will prove
that for non-integer $a>0$, if
the weak derivative $\nabla^\alpha\hat f$ is in $L^1(\T^d)$ for all multi-indices $\alpha$ with $\abs \alpha = \floor a$, and if in addition  there exists $\eta \in ( a - \floor a, 1)$ such that
(with $\tilde u=(u,0,\ldots,0)\in \R^d$)
\begin{align}
\label{eq:Uf-intro}
\bignorm{ \nabla^\alpha[\hat f(\cdot +\tilde u)- \hat f(\cdot -\tilde u)]}_{L^1(\T^d)}
\lesssim u^\eta
\qquad (0\le u\le1)
\end{align}
for all $\abs \alpha = \floor a$, then $f(x)$ decays at least as fast as $|x|^{-a}$.
We apply this to obtain the error estimate stated in Theorem~\ref{thm:gaussian_lemma}.

\subsection{Proof of Theorem~\ref{thm:gaussian_lemma} with integer-power error estimate}
\label{sec:nn-weak}

We first prove a version of Theorem~\ref{thm:gaussian_lemma} with an integer power of $\abs x$ in the error term, which is in fact sufficient for applications.
The key ideas in the paper are present already here.
In Section~\ref{section:proof_frac}
we improve the error estimate
using fractional derivatives and complete the proof of Theorem~\ref{thm:gaussian_lemma}.

We assume that $\rho > \frac{d-8}{2} \vee 0$.
Let
\begin{align}
\label{eq:nddef}
n_d = d-2+s_0
\quad \text{with} \quad
s_0 = \begin{cases}
    0 & (\rho \le 1 + (\frac{d-8}{2} \vee 0) )     \\
    1 & (\rho > 1+  (\frac{d-8}{2} \vee 0)),
    \end{cases}
\end{align}
\ie, we pick the largest integer $s_0$ that satisfies \eqref{eq:s_small}--\eqref{eq:s_large}.
By definition, and since $\rho > \frac{d-8}{2} \vee 0$,
\begin{equation}
\label{eq:ndub}
    n_d < (d - 2 + \rho \wedge 2 ) \wedge ( \tfrac 1 2 d + 2 + \rho ).
\end{equation}

\begin{theorem}
\label{thm:gaussian_lemma_int}
Let $d>2$, and let $F$ satisfiy Assumption~\ref{ass:F}. Then
the solution $G$ to \eqref{eq:FG},
given by the Fourier integral \eqref{eq:Gint},
satisfies
\begin{align}
\label{eq:G_asymp_int}
    \qquad~
    G(x) =
    \lambda C_\mu(x) +  o\(\frac 1 {\abs x^{n_d}}\)
    \quad \text{as $\abs x \to \infty$}.
\end{align}
\end{theorem}

\subsubsection{Fourier analysis}
\label{sec:prelim}

Let $F$ satisfy Assumption~\ref{ass:F}.
As in \eqref{eq:G_isolate-intro}, we write
\begin{align}
G
&= \lambda C_\mu + f,
\qquad
f = f_{\lambda,\mu} =  C_\mu * E_{\lambda,\mu} * G,
\end{align}
where $\lambda$ and $\mu$ have been chosen as in \eqref{eq:lambda_z} in order to
ensure that $E_{\lambda,\mu}$ has vanishing zeroth and second moments.
The following lemma translates this good moment behaviour
into good bounds on $\hat E_{\lambda, \mu}$ and its derivatives, which are fundamental in our proof.
Here and subsequently we omit the subscripts $\lambda, \mu$ and
use subscripts to denote partial derivatives instead, e.g., $\hat E_\alpha = \nabla^\alpha \hat E_{\lambda, \mu}$.

\begin{lemma}
\label{lemma:Ek}
Suppose $E: \Z^d \to \R$ is $\Z^d$-symmetric,
has vanishing zeroth and second moments as in
\eqref{eq:E_condition},
and satisfies $\abs{ E(x) } \le  K \abs x^{-(d+2 + \rho)}$ for some $K, \rho>0$.
Choose $\sigma \in (0, \rho)$ such that $\sigma \le 2$ and let $\alpha$ be a multi-index
with
$\abs \alpha < 2 + \sigma$.
Then there is a constant $c = c(\sigma, \rho, d)$ such that
\begin{align} \label{eq:Ek}
\abs{ \hat E_\alpha (k) }
\le cK \, \abs k^{2+\sigma - \abs \alpha}.
\end{align}
\end{lemma}

\begin{proof}
Let $g_x(k) = \cos(k\cdot x) - 1 + \frac{ (k\cdot x)^2 }{2!}$.
Since $E$ is even, $\sum_{x\in\Z^d} x_i x_j E(x) = 0$ for $i\ne j$. By \eqref{eq:E_condition} and by  symmetry, $\sum_{x\in\Z^d} x_i^2 E(x) = 0$ for all $i$.
Therefore,
\begin{align}
\hat E(k) = \sum_{x\in\Z^d} E(x) \cos(k\cdot x) = \sum_{x\in\Z^d} E(x) g_x(k).
\end{align}
By explicit differentiation of $g_x(k)$, together with
inequalities such as $\abs{ \cos t - 1 + \frac 12 t^2} \lesssim \abs t^{2+\tau}$
for any $\tau\in [0, 2]$, we find that
\begin{equation}
     |g_x(k)| \lesssim |k\cdot x|^{2+\sigma}  ,
    \qquad
    |\nabla_i g_x(k)| \lesssim |x|\,|k\cdot x|^{1+\sigma},
    \qquad
     |\nabla_{ij}^2 g_x(k)| \lesssim |x|^2 |k\cdot x|^{\sigma} .
\end{equation}
The case $|\alpha|=3$ can only occur if $\sigma\in (1,2]$, and in this case
\begin{equation}
    |\nabla_{ijl} g_x(k)| = |x_ix_jx_l\sin(k\cdot x)|
    \lesssim |x|^3 |k\cdot x|^{\sigma-1}.
\end{equation}
In all cases, we have
$\abs{ \grad^\alpha g_x(k) } \le c_{{0}} \abs k^{2+\sigma - \abs \alpha} \abs x^{2+\sigma}$, so
\begin{align}
\abs{ \hat E_\alpha (k) }
&\le \sum_{x\in\Z^d} \abs {E(x)} \abs{ \grad^\alpha g_x(k) }
\le c_{{0}}K \, \abs k^{2+\sigma - \abs \alpha} \sum_{x\ne 0}  \frac {\abs x^{2+\sigma}} {\abs x^{d+2+\rho}}.
\end{align}
The sum on the right-hand side converges since $\sigma <\rho$, and the proof is complete.
\end{proof}

For a function $g : \T^d \to \C$
and $p \in [1,\infty)$, we denote $L^p(\T^d)$ norms with a subscript $p$:
\begin{equation}
    \norm g _p^p = \int_{\T^d} |g(k)|^p \frac{dk}{(2\pi)^d} ,
\end{equation}
and we write $\|g\|_\infty$ for the supremum norm.

We will use the following general fact which relates the smoothness of a function
on $\T^d$ to the decay of its Fourier coefficients.
It involves the notion of \emph{weak derivative},
which extends the classical notion of derivative to locally integrable functions.
The weak derivative is defined via the usual integration by parts formula,
and for this reason the function and its derivative are required to be locally integrable.
We work on the compact space $\T^d$, so weak derivatives are automatically 
in $L^1(\T^d)$ when they exist.
Other basic facts concerning weak derivatives are recalled in Appendix~\ref{appendix:weak}.
It is by using weak derivatives that we overcome the obstacles to naive application
of the ideas in \cite{Slad22_lace} to percolation and LTLA.
In the proof of Lemma~\ref{lem:Graf}, and elsewhere,
for a multi-index $\alpha$ and for $x=(x_1,\ldots,x_d) \in \Z^d$ we use the notation
$x^\alpha = \prod_{i=1}^d  x_i^{\alpha_i}$.

\begin{lemma}
\label{lem:Graf}
Let $n,d>0$ be positive integers.
Suppose that $\hat h : \T^d \to \C$ is  $n$ times weakly differentiable,
and let $h : \Z^d \to \C$
be the inverse Fourier transform of $\hat h$.
There is a constant $c_{d,n}$ depending only on the dimension $d$ and the
order $n$ of differentiation, such that
\begin{equation}
\label{eq:Graf}
    |h(x)|
    \le
    c_{d,n} \frac{1}{\xvee^n}
    \max_{|\alpha| \in \{0,n\}}\|\hat h_\alpha\|_1 .
\end{equation}
Moreover, $|x|^n h(x) \to 0$ as $|x|\to\infty$.
\end{lemma}

\begin{proof}
The bound \eqref{eq:Graf} follows from
\cite[Corollary~3.3.10]{Graf14}.
Although stated for classical derivatives, the proof of \cite[Corollary~3.3.10]{Graf14}
applies for weak derivatives since it only uses integration by parts.
Also, it follows from the Riemann--Lebesgue Lemma
\cite[Proposition~3.3.1]{Graf14} and
the integrability of $\hat h_\alpha$ that the inverse Fourier transform of
$\hat h_\alpha$, which is a constant multiple of $x^\alpha h(x)$, vanishes at infinity.
Since this holds for any $\abs \alpha \le n$, we obtain $ \abs x^n h(x) \to 0$ as $\abs x \to \infty$.
\end{proof}

\begin{proposition}
\label{prop:f_NN}
Let $F$ obey Assumption~\ref{ass:F}.
Then the function $\hat f = \hat C \hat E \hat G$ is $n_d$ times weakly differentiable.
\end{proposition}

\begin{proof}[Proof of Theorem~\ref{thm:gaussian_lemma_int} assuming Proposition~\ref{prop:f_NN}]
Since $\hat f$ is $n_d$ times weakly differentiable by Proposition~\ref{prop:f_NN},
it follows from Lemma~\ref{lem:Graf} that
$|x|^{n_d} f(x) \to 0$ as $\abs x\to \infty$.
The equation $G = \lambda C_\mu + f$ then gives
\begin{align}
G(x) =
\lambda C_{{\mu}}(x) + o\( \frac 1 {\abs x^{n_d}}\)
    \quad \text{as $\abs x \to \infty$},
\end{align}
which completes the proof.
\end{proof}

Recall that $\hat A = 1-\mu\hat\Dnn$.
To prove Proposition~\ref{prop:f_NN},
we must study derivatives of $\hat f = \hat C \hat E \hat G = \hat E / (\hat A \hat F)$.
For this, we would like to use the product and quotient rules of differentiation
to conclude that $\hat f_\alpha$ is given by a linear combination of terms of the form
\begin{align} \label{eq:f_decomp0}
\frac{ \prod_{n=1}^i \hat A_{\delta_n} }{\hat A ^{1+i} }
\hat E_{\alpha_2}
\frac{ \prod_{m=1}^j \hat F_{\gamma_m}  }{\hat F^{1+j} }
=
\( \prod_{n=1}^i \frac{ \hat A_{\delta_n} }{\hat A } \)
\( \frac{ \hat E_{\alpha_2} }{ \hat A \hat F } \)
\( \prod_{m=1}^j \frac{ \hat F_{\gamma_m}  }{ \hat F } \),
\end{align}
where $\alpha = \alpha_1 + \alpha_2 + \alpha_3$, $ 0 \le i \le \abs {\alpha_1}$, $0 \le j \le \abs {\alpha_3 }$, $\sum_{n=1}^i \delta_n = \alpha_1$, and $\sum_{m=1}^j \gamma_m = \alpha_3$. However, the use of calculus rules to compute $\hat f_\alpha$ requires justification.
Weak derivatives satisfy a version of product and quotients rules stated in Appendix~\ref{appendix:weak}.
These imply that if a naive application
of the product and quotient rules produces a product which {\it a posteriori} can
be seen to be integrable via H\"older's inequality, then the
weak derivative indeed exists and is given by the usual product and quotient rules.
The following lemma is what allows us to use H\"older's inequality
in this manner.

\begin{lemma}
\label{lem:AFEbds}
Let $F$ obey Assumption~\ref{ass:F}.
Let $\gamma$ be a multi-index with $ \abs \gamma < (d - 2 + \rho \wedge 2 ) \wedge ( \tfrac 1 2 d + 2 + \rho )$,
and choose $\sigma \in (0, \rho)$ such that $\sigma \le 2$.
Choose $q_1,q_2$ satisfying
\begin{equation}
    \frac{ \abs \gamma } d < q_1\inv
    ,
    \qquad
    \frac{ 2 - \sigma + \abs \gamma } d < q_2\inv
    .
\end{equation}
Then $\hat F, \hat A, \hat E$ are $\abs \gamma$ times weakly differentiable,
and
\begin{align}
\label{eq:FFAA}
\frac{ \hat A_\gamma } { \hat A } , \;
\frac{ \hat F_\gamma } { \hat F }
	\in L^{q_1}(\T^d) ,	\qquad
\frac{ \hat E_{\gamma} }{\hat A \hat F} \in L^{q_2}(\T^d).
\end{align}
\end{lemma}

\begin{proof}[Proof of Proposition~\ref{prop:f_NN} assuming Lemma~\ref{lem:AFEbds}]
Let $\abs \alpha \le n_d$,
and pick some $\sigma \in (0, \rho \wedge 2)$.
We use the quotient and product rules (Lemmas~\ref{lemma:product_rule} and \ref{lemma:quotient_rule}) to calculate $\hat f_\alpha$, and we use Lemma~\ref{lem:AFEbds} to verify the hypotheses of the rules.
In order to apply Lemma~\ref{lem:AFEbds} with $n_d$ derivatives
we require that
\begin{equation}
\label{eq:ndub-bis}
    n_d < (d - 2 + \rho \wedge 2 ) \wedge ( \tfrac 1 2 d + 2 + \rho ),
\end{equation}
which is guaranteed by \eqref{eq:ndub}. 

We need to verify that all terms of the form \eqref{eq:f_decomp0} are integrable to conclude that $\hat f_\alpha$ exists and is given by a linear combination of \eqref{eq:f_decomp0}.
By Lemma~\ref{lem:AFEbds} and 
H\"older's inequality, the $L^r$ norm of \eqref{eq:f_decomp0} is bounded by
\begin{align}
\( \prod_{n=1}^i \Bignorm{ \frac{ \hat A_{\delta_n} }{\hat A } }_{r_n} \)
\Bignorm{ \frac{ \hat E_{\alpha_2} }{ \hat A \hat F }  }_q
\( \prod_{m=1}^j \Bignorm{ \frac{ \hat F_{\gamma_m}  }{ \hat F } }_{q_m} \),
\end{align}
where $r$ can be any positive number that satisfies
\begin{align}
\frac 1 r &= \( \sum_{n=1}^i \frac 1 {r_n} \) + \frac 1 q + \( \sum_{m=1}^j \frac 1 {q_m} \) \nl
&> \frac{ \sum_{n=1}^i \abs{\delta_n}  } d  + \frac{ 2 - \sigma + \abs{\alpha_2}   } d + \frac{ \sum_{m=1}^j \abs{\gamma_m}  } d
= \frac{  \abs \alpha  + 2 - \sigma } d.
\end{align}
Since $\sigma < \rho \wedge 2$ is arbitrary, $\eqref{eq:f_decomp0}$ is in $L^r$ for all
$r\inv > (\abs \alpha + 2 - \rho \wedge 2 )/d$.
We can take $r=1$ because $\abs \alpha \le n_d < d - 2 + \rho \wedge 2$ by \eqref{eq:ndub-bis}.
This proves that
the weak derivative $f_\alpha$ exists
and that
\begin{align} \label{eq:f_alpha}
\hat f_\alpha \in L^r( \T^d )
	\qquad (r\inv > \frac{ \abs \alpha + 2 - \rho \wedge 2}d ).
\end{align}
This completes the proof.
\end{proof}

\subsubsection{Proof of Lemma~\ref{lem:AFEbds}}
\label{sec:nn-pf-weaker}

To complete the proof of Theorem~\ref{thm:gaussian_lemma_int}, it
remains to prove Lemma~\ref{lem:AFEbds}.
The proof uses the following elementary lemma.

\begin{lemma}
\label{lem:FTbdd}
Let $h: \Z^d \to \R$ obey $|h(x)| \le K \nnnorm x^{-b}$ for some $K,b>0$.
\begin{enumerate}[label=(\roman*)]
\item
If $b>d$ then $h \in \ell^1(\Z^d)$, $\hat h \in L^\infty(\T^d)$, and $\norm{ \hat h }_\infty \le c_{d,b} K$.

\item
If $b\le d$ then $h \in \ell^p(\Z^d)$ for $p> \frac d b$.
If also $\frac d2 < b \le d$ then $\hat h \in L^q(\T^d)$
and $\norm{ \hat h } _q \le c_{d,b,q}K$  for all $1 \le q < \frac d {d-b}$.
\end{enumerate}
\end{lemma}

\begin{proof}
(i) This follows from the summability of $\nnnorm x^{-b}$.

\smallskip\noindent
(ii)
The fact that $h \in \ell^p(\Z^d)$ for $p> \frac d b$ follows from the summability
of $\nnnorm x^{-pb}$ for $pb>d$.
For the second claim,
the condition $b > \frac d 2$ ensures that $h \in \ell^2(\Z^d)$,
so that we can use the $L^2$ Fourier transform.
By the Hausdorff--Young inequality (see, e.g., \cite[Theorem~4.28]{Foll15}),
the Fourier transform is a bounded linear map from $\ell^p(\Z^d)$ to $L^q(\T^d)$
when $p \in [1,2]$ and $p \inv + q \inv = 1$.
In particular,
we get $\hat h \in L^q(\T^d)$ with  $2 \le q<\frac{d}{d-b}$.
The extension to
$1 \le q<\frac{d}{d-b}$ follows from monotonicity of the $L^q(\T^d)$ norm in $q$.
\end{proof}

\begin{proof}[Proof of Lemma~\ref{lem:AFEbds}]
\emph{Bound on $\hat A_\gamma / \hat A$.}
Let $\gamma$ be any multi-index.
By definition, $\hat A=1-\mu\hat\Dnn$, so $\hat A$
is infinitely classically differentiable.
There is nothing to prove for $\abs \gamma=0$ since the ratio is then $1$.
If $\abs \gamma = 1$, by Taylor's theorem and symmetry, we have $\abs { \hat A_\gamma (k) } \lesssim \abs k$. If $\abs \gamma \ge 2$, Taylor's theorem gives $\abs { \hat A_\gamma (k) } \lesssim 1$ instead.
It then follows from the infrared bound \eqref{eq:IR-SRW} that
\begin{align} \label{eq:AA_k}
\Bigabs{ \frac{ \hat A_\gamma }{ \hat A } (k) }
\lesssim  \abs k^{-(\abs \gamma \wedge 2)}
\in L^q(\T^d)	\qquad (q\inv > \frac{ \abs \gamma \wedge 2 } d),
\end{align}
which is stronger than the desired \eqref{eq:FFAA}.

\medskip\noindent
\emph{Bound on $\hat F_\gamma / \hat F$.}
Let $ \abs \gamma < (d-2+\rho) \wedge (\tfrac 1 2 d  + 2 +\rho)$.
There is again nothing to prove for $\abs \gamma =0$.
If $\abs \gamma = 1$, by Taylor's theorem and symmetry, and since $\sum_{x\in\Z^d} |x|^2|F(x)|$ is finite, we have $\abs { \hat F_\gamma (k) } \lesssim \abs k$.
It then follows from the infrared bound that
$\abs{ \hat F_\gamma (k)/ \hat F  (k) } \lesssim \abs k \inv  \in L^q$ for $q\inv > 1/d$.

It remains to consider the case $\abs\gamma \ge 2$.
By the decay hypothesis on $F(x)$,
the condition $\abs \gamma < \tfrac 1 2 d  + 2 +\rho$ precisely ensures that $\abs x^{ \abs \gamma } F(x) \in \ell^2(\Z^d)$,
so that Lemma~\ref{lemma:weak_Fourier} applies and
shows that $\hat F$ is $\abs \gamma$ times weakly differentiable.
We next estimate $\hat F_\gamma$
using Lemma~\ref{lem:FTbdd} with $h(x) = x^\gamma F(x)$ and
\begin{align} \label{eq:bgood}
b=- \abs \gamma +d+2+\rho .
\end{align}
If $2 \le \abs \gamma <2+\rho$ then $b > d$, and Lemma~\ref{lem:FTbdd}(i) gives $\hat F_\gamma  \in L^\infty$.
On the other hand,
if $2+\rho \le \abs \gamma < \frac 1 2 d + 2 + \rho$ then $b \in ( \frac d 2, d]$,
and Lemma~\ref{lem:FTbdd}(ii) gives $\hat F_\gamma \in L^r$ whenever $ r \inv > (\abs \gamma - 2 - \rho)/d$.
Since $\rho>0$,
we can combine the two subcases to get
\begin{align}
\label{eq:F_gamma}
\hat F_\gamma \in L^{d/(\abs \gamma - 2)}(\T^d) \qquad (2 \le \abs \gamma < \tfrac 1 2 d + 2 +\rho ).
\end{align}
Then,
by the infrared bound, it follows from H\"older's inequality that
\begin{equation}
    \Bignorm{ \frac{ \hat F_{\gamma}  }{ \hat F } }_{q}
    \lesssim \Bignorm{ \frac{1}{|k|^2} }_p
    \norm{\hat F_{\gamma}  }_{\frac{d}{|\gamma|-2}}
    < \infty
\end{equation}
provided that $q^{-1} = p^{-1}+  \frac{|\gamma|-2}{d} > \frac 2d +\frac{|\gamma|-2}{d}
=
\frac{|\gamma|}{d}$. This gives the desired result.

\medskip\noindent
\emph{Bound on $\hat E_\gamma / \hat A \hat F$.}
Let $ \abs \gamma <  (d-2+\rho) \wedge (\tfrac 1 2 d  + 2 +\rho)$.
We divide into two cases, now according to how $\abs \gamma$ compares with $2 + \sigma$.
We use the fact that $E = A - \lambda F$ has the same $|x|^{-(d+2+\rho)}$ decay as $F$.

If $\abs \gamma < 2 + \sigma$, Lemma~\ref{lemma:Ek}
and the infrared bounds give $\abs{ \hat E_\gamma / (\hat A \hat F) (k)}
\lesssim \abs k^{2 + \sigma - \abs \gamma - 4} = \abs k^{-(2 - \sigma + \abs \gamma)}$,
which is in $L^q$ for $q\inv > (2 - \sigma  + \abs \gamma)/d$.

If $2 + \sigma \le \abs \gamma < \half d + 2 + \rho$, as in the proof for $\hat F_\gamma /\hat F$, from the decay of $E(x)$ we get $\hat E_\gamma \in L^\infty$ when $ 2 + \sigma \le \abs \gamma < 2 + \rho$, and $\hat E_\gamma \in L^r$ for $r\inv > (\abs \gamma - 2 - \rho)/d$ when $ 2 + \rho \le \abs \gamma < \half d + 2 + \rho$.
Now we use $\sigma < \rho$ to combine the two subcases as
\begin{align}
\hat E_\gamma \in L^r(\T^d) 	\qquad (r\inv > \frac{ \abs \gamma - 2 - \sigma }d ).
\end{align}
The desired result then follows from H\"older's inequality and the infrared bound $ \abs{ (\hat A \hat F)\inv (k) } \lesssim \abs k^{-4} \in
L^p$ for all $p\inv > 4/d$.
\end{proof}

For later use, we note that the reasoning
used to bound $\hat E_\gamma / \hat A \hat F$ gives,
for $\sigma \le \abs \gamma < \half d + 2 + \rho$,
\begin{align} \label{eq:EAEF}
\frac{ \hat E_{\gamma} }{ \hat A},\; \frac{ \hat E_{\gamma} }{\hat F } 	
\in L^q(\T^d)
	\qquad  (q\inv > \frac{ \abs{\gamma} - \sigma   }d).
\end{align}

\begin{remark} \label{rk:rho_restriction}
In the paragraph containing \eqref{eq:bgood}, we used
$\abs x^{\abs \gamma} F(x) \in \ell^2(\Z^d)$ in order to invoke boundedness
of the $L^2$ Fourier transform (Lemma~\ref{lemma:weak_Fourier}).
In the proof of Proposition~\ref{prop:f_NN} we need this for $|\gamma|=n_d \ge d-2$.
Our unnatural condition $\rho>\frac{d-8}{2}$ is present  
to make
\begin{align}
d-2 < \tfrac 12 d + 2 + \rho,
\end{align}
so that the inequality \eqref{eq:ndub},
which ensures that $|\gamma|=n_d$ obeys the hypothesis of Lemma~\ref{lem:AFEbds},
can be satisfied.
\end{remark}

\subsection{Improved error estimate for Theorem~\ref{thm:gaussian_lemma}}
\label{section:proof_frac}

We now improve the error estimates of Section~\ref{sec:nn-weak} using
a notion of
``fractional derivative.''

\subsubsection{Fractional derivatives}
\label{sec:fd}

There are various notions of fractional derivatives in the literature.
Given $\epsilon \in (0,1)$, the fractional derivative of a power series
$f(z) = \sum_{n=0}^\infty a_nz^n$ was defined in
\cite[Section~3.2]{HS92a} to be $\sum_{n=0}^\infty n^\epsilon a_n z^n$, and
an integral representation for the fractional derivative was presented.
The integral representation provides a way to estimate the fractional
derivative using only the standard first derivative.

Our error bound in Theorem~\ref{thm:gaussian_lemma} requires a constant
upper bound on $|x|^{d-2+s} f(x)$, and when $d-2+s$ is not an integer the
method of Section~\ref{sec:nn-weak}
bounds only $\abs x^{d-2 + \floor s} f(x)$.
To improve,
we use an integral representation for
the Fourier transform of $({\rm sgn}\,x_1) |x_1|^\delta g(x)$,
where $\delta = s - \floor s$ and $g(x) = \abs x^{d-2 - \floor s} f(x)$.

We begin with an integral formula for a fractional power of $t \ge 0$
that can be proved by a change of variables, namely
that for $\delta \in (0,1)$,
\begin{align}
\label{eq:tdelta}
    t^\delta =
    \frac{1}{c_\delta} \int_0^\infty \frac{ \sin (tu) }{ u^{1+\delta}} du
    \quad
    \text{with}
    \quad
c_\delta = \int_0^\infty \frac{ \sin u }{ u^{1+\delta}} du \in (0, \infty).
\end{align}
The restriction that $\delta \in (0,1)$ ensures integrability in \eqref{eq:tdelta}.
Given a function $\hat g: \T^d \to \C$, we regard it as a periodic function on $\R^d$ and define
\begin{equation}
    (U_u\hat g)(k) = \hat g(k+\uvec) - \hat g(k-\uvec)
\end{equation}
with $k \in \R^d$ and $\uvec = (u, 0, \dots, 0) \in \R^d$.
The following lemma provides an integral representation for the Fourier
transform of $({\rm sgn}\,x_1) |x_1|^\delta g(x)$.

\begin{lemma}
\label{lem:fracder}
Let $\delta \in (0,1)$.
Suppose that $\hat g \in L^1(\T^d)$ and that $\int_0^\infty u^{-(1+\delta)}\|U_u\hat g\|_1 du <
\infty$.  Let
\begin{equation} \label{eq:Uu_int}
    \hat w(k) = \frac{1}{2ic_\delta} \int_0^\infty \frac{1}{u^{1+\delta}}(U_u\hat g)(k) du.
\end{equation}
Then the inverse Fourier transform of $\hat w(k)$ is $w(x)=({\rm sgn}\, x_1)|x_1|^\delta g(x)$,
where $g$ denotes the inverse Fourier transform of $\hat g$.
\end{lemma}

\begin{proof} 
By hypothesis, $\hat w \in L^1(\T^d)$, so its inverse Fourier transform is
\begin{align}
    w(x) & = \int_{\T^d} \hat w(k) e^{-ik\cdot x} \frac{dk}{(2\pi)^d}
    =
    \frac{1}{2ic_\delta} \int_0^\infty \frac{du}{u^{1+\delta}}
    \int_{\T^d} \frac{dk}{(2\pi)^d} (U_u\hat g)(k)e^{-ik\cdot x},
\end{align}
where the application of Fubini's Theorem is justified by hypothesis.
A change of variables to evaluate the $k$-integral gives
\begin{align}
    w(x) & =
    \frac{1}{c_\delta} g(x) \int_0^\infty \frac{du}{u^{1+\delta}}
    \sin(ux_1),
\end{align}
and \eqref{eq:tdelta} then gives $w(x) = ({\rm sgn}\, x_1)|x_1|^\delta g(x)$.
\end{proof}

Integral representations similar to but not identical
to \eqref{eq:Uu_int}   
have been employed previously,
\eg, in \cite[Section~4.2]{CS09a} and \cite[Section~3]{DPV12}.
In the latter reference, the fractional Laplace operator on $\R^d$ is defined, for $\beta \in (0,1)$ and $\hat g (\cdot)$ a Schwartz function, by
\begin{align} \label{eq:fractional_laplace}
(-\Delta)^\beta \hat g(k) = - \frac 1 {2 c'_\beta} \int_{\R^d} \frac{ \hat g(k + u) - 2 \hat g(k) + \hat g( k - u ) } { \abs{ u }^{d + 2\beta } } du ,
	\qquad c'_\beta = \int_{\R^d} \frac{ 1 - \cos(u_1) } { \abs{ u }^{d + 2\beta } } du .
\end{align}
We have used the first-order difference operator $U_u$ in \eqref{eq:Uu_int}, instead of the second-order difference operator in \eqref{eq:fractional_laplace}.
Also, our variable $u$ in \eqref{eq:Uu_int} is one-dimensional as we are only modifying  the first coordinate direction,
while the fractional Laplacian corresponds to $\abs x^{2\beta}$.

The following lemma improves Lemma~\ref{lem:Graf} by a fractional power of $\xvee$.

\begin{lemma}
\label{lemma:v_bound_frac}
Let $d >0$ and let $h: \Z^d \to \C$ be
the inverse Fourier transform of $\hat h$.
Suppose that $h$ is $\Z^d$-symmetric,
that $a>0$ is not an integer, and that
$\hat h$ is $\lfloor a \rfloor$ times weakly differentiable.
Suppose also that, for some $\eta \in ( a - \floor a, 1)$
and
for all $\abs \alpha =\lfloor a \rfloor$,
there is a $K_\alpha >0$ such that
\begin{align}
\label{eq:Uv_alpha}
\norm{ U_u \hat h_\alpha }_{1} \le u^\eta K_\alpha  \qquad (0\le u\le1).
\end{align}
Then there is a constant $c_{d,a,\eta}$ depending only on $d$, $a$, and $\eta$ such that
\begin{align}
\label{eq:v_bound_frac}
    \abs{ h(x) }
    \le
    c_{d,a,\eta} \frac 1 {\xvee^{a}} \Big(\|\hat h\|_1+
    \max_{\abs \alpha = \floor a} ( K_\alpha + \norm{ \hat h_\alpha }_{1} )
    \Big).
\end{align}
\end{lemma}

\begin{proof}
We write $\delta = a - \floor a \in ( 0,\eta )$.
For $x=0$, there is nothing to prove since the bound $\abs{ h(0) } \le \norm{ \hat h }_1$ implies \eqref{eq:v_bound_frac}.
Let $x \neq 0$ and $|\alpha|=\lfloor a \rfloor$. We first use the hypothesis to observe that
\begin{align}
    \int_0^\infty \frac{1}{u^{1+\delta}}\|U_u\hat h_\alpha\|_1 du
    & \le
    K_\alpha \int_0^1 \frac{u^\eta}{u^{1+\delta}}  du
     +
     2 \int_1^\infty \frac{1}{u^{1+\delta}}\|\hat h_\alpha\|_1 du
     \lesssim
     K_\alpha + \|\hat h_\alpha\|_1.
\end{align}
Let $\hat w(k) = \frac{1}{2ic_\delta}\int_0^\infty \frac{1}{u^{1+\delta}} (U_u\hat h_\alpha)(k) du$.
The above inequality shows that $\|\hat w\|_1 \lesssim K_\alpha + \| h_\alpha\|_1$.
Since the inverse Fourier transform of $\hat h_\alpha$ is $(-ix)^\alpha h(x)$,
it follows from Lemma~\ref{lem:fracder} that
\begin{equation}
    |x_1|^\delta |x^\alpha h(x)| = |w(x)| \le \|\hat w\|_1 \lesssim K_\alpha + \| h_{\alpha}\|_1.
\end{equation}
By symmetry,
the above also holds with $x_1$ replaced by $x_j$ for any $j$.
Since $a=|\alpha| +\delta$, we have $|x|^a|h(x)|\lesssim
\max_{|\alpha|=\lfloor a \rfloor}(
K_\alpha + \| h_{\alpha}\|_1)$.
Together with the bound for $x=0$,  this completes the proof.
\end{proof}

When it is possible to take a full derivative of $\hat h_\alpha$,
we can verify
\eqref{eq:Uv_alpha}
using the next lemma,
which can be considered
as an estimate on the fractional difference quotient $u^{-\eta}U_ug$.
For $\R^d$, the case $\eta=1$ can be found, e.g., in \cite[Sec.~5.8.2, Theorem~3]{Evan10}.

\begin{lemma} \label{lemma:h_diff}
Let $g: \T^d \to \C$ be weakly differentiable.  Fix $1\le p < d$.
Assume $\grad_j g \in L^p(\T^d)$ for all $j$.
Let $0 \le \eta \le 1$ and define $p_\eta$ by $\frac 1 { p_\eta } = \frac 1 p - \frac { 1- \eta} d$. Then \begin{align}
\norm{ U_u g }_{p_\eta } \lesssim u^\eta \norm{ g }_{W^{1,p} },
\end{align}
where $\norm{ g }_{W^{1,p} } = ( \norm g_p^p + \sum_{j=1}^d \norm{ \grad_j g }_p^p )^{1/p}$.
\end{lemma}

\begin{proof}
Define $p^*=p_0$
by $\frac 1 { p^* } = \frac 1 p - \frac { 1} d$, so that $\frac 1 {p_\eta} = \frac \eta p + \frac{ 1-\eta }{p^*}$.
Since $\norm{ U_u g }_{p_\eta } \le \norm{ U_u g }_{p }^\eta \norm{ U_u g }_{p^* }^{1-\eta}$ by the log-convexity of $L^p$ norms, it suffices to prove the cases $\eta = 0, 1$.
The case $\eta = 0$ follows from the triangle inequality and the torus Sobolev inequality
\cite[Corollary~1.2]{BO13}:
\begin{align}
\norm{ U_u g }_{p^* } \le 2 \norm{ g }_{p^* }
\lesssim
\norm{ g }_{W^{1,p} }.
\end{align}
For $\eta =1$, since $C^1$ is dense in $W^{1,p}$ it suffices to consider
$g \in C^1$, for which the Fundamental Theorem of Calculus gives
\begin{align}
U_u g(k) = u \int_{-1}^1 \grad_1 g (k + t \uvec) dt.
\end{align}
Then by Minkowski's integral inequality,
\begin{align} \label{eq:FTCbdd}
\norm{ U_u g }_p \le u \int_{-1}^1 \norm{ \grad_1 g (\cdot + t \uvec) }_p dt
= 2u \norm{ \grad_1 g  }_p
\le 2u \norm{ g }_{W^{1,p} }.
\end{align}
This completes the proof.
\end{proof}

\subsubsection{Proof of Theorem~\ref{thm:gaussian_lemma} }
\label{sec:proof_thm_strong}

With Lemma~\ref{lemma:v_bound_frac}, we reduce the proof of Theorem~\ref{thm:gaussian_lemma} to the following proposition.

\begin{proposition} \label{prop:f_frac}
Let $\abs \alpha \le n_d$, $0 \le u \le 1$, and suppose $\eta \in (0,1)$ satisfies
\begin{align} \label{eq:5.5}
n_d + \eta < (d - 2 + \rho \wedge 2)) \wedge ( \tfrac 1 2 d + 2 + \rho ) .
\end{align}
If $F$ obeys Assumption~\ref{ass:F}, then
\begin{align}
\norm{ U_u \hat f_\alpha }_1 \lesssim u^\eta.
\end{align}
\end{proposition}

\smallskip
\begin{proof}[Proof of Theorem~\ref{thm:gaussian_lemma} assuming Proposition~\ref{prop:f_frac}]
By definition, $s \in [0,\rho\wedge 2 \wedge (\rho-\frac{d-8}{2})) $.
It suffices to prove the theorem for values of $s$ close to the upper
limit of this interval, so we assume $s$ is not an integer.
Let $\delta = s - \lfloor s \rfloor \in (0,1)$.  Then
by \eqref{eq:nddef} (where $\lfloor s \rfloor$ is written as $s_0$),
\begin{equation} \label{eq:slimits}
    d-2+s = n_d+\delta < (d - 2 + \rho\wedge 2) \wedge ( \tfrac 1 2 d + 2 + \rho ).
\end{equation}
Since we have a strict inequality, we can pick an $\eta \in (\delta, 1)$
satisfying \eqref{eq:5.5}.
By Proposition~\ref{prop:f_frac}, the error term $f$ in the decomposition
$G = \lambda C_\mu + f$ obeys $\norm{ U_u \hat f_\alpha }_1 \lesssim u^\eta$.
With Proposition~\ref{prop:f_NN}, this supplies the hypothesis of
Lemma~\ref{lemma:v_bound_frac} (with $a=n_d+\delta$), whose conclusion gives  the desired result that
$f(x) = O(\nnnorm x^{-(d-2+s)})$.
\end{proof}

To prove Proposition~\ref{prop:f_frac}, as in \eqref{eq:f_decomp0}, we write $\hat f_\alpha$ as a linear combination of terms of the form
\begin{align}
\label{eq:f_decomp}
\( \prod_{n=1}^i \frac{ \hat A_{\delta_n} }{\hat A } \)
\( \frac{ \hat E_{\alpha_2} }{ \hat A \hat F } \)
\( \prod_{m=1}^j \frac{ \hat F_{\gamma_m}  }{ \hat F } \).
\end{align}
To calculate $U_u \hat f_\alpha$, we use the linearity of $U_u$ and
apply the following elementary formula for the difference of two products. Let $X_{k,l}=\prod_{i=k}^l a_i$ and $Y_{k,l}=\prod_{i=k}^l b_i$.  Then, with
empty products equal to $1$, \begin{equation}
\label{eq:X-Y}
    X_{1,n} - Y_{1,n}
    =
    \sum_{j=1}^n (a_j-b_j)X_{1,j-1}Y_{j+1,n}.
\end{equation}
We apply this to \eqref{eq:f_decomp}, so that $U_u$ acts on one factor at a time.
We will prove the following lemma.

\begin{lemma} \label{lemma:Uterms}
Let $\gamma$ be a multi-index, $0 \le \eta \le 1$, with $\abs \gamma + \eta <  (d - 2 + \rho \wedge 2 ) \wedge ( \tfrac 1 2 d + 2 + \rho )$.  
Choose $\sigma \in (0, \rho)$ such that $\sigma \le 2$, and choose $q_1,q_2$ satisfying
\begin{equation}
    \frac{ \abs \gamma + \eta} d < q_1\inv
    ,
    \qquad
    \frac{ 2 - \sigma + \abs \gamma +\eta } d < q_2\inv
    .
\end{equation}
Then,
for $F$ obeying Assumption~\ref{ass:F} and for
$0 \le u \le 1$,
\begin{align} \label{eq:Uterms}
\Bignorm{ U_u ( \frac{ \hat A_\gamma } { \hat A } ) }_{q_1}, \;
\Bignorm{ U_u ( \frac{ \hat F_\gamma } { \hat F } ) }_{q_1}, \;
\Bignorm{ U_u ( \frac{ \hat E_\gamma } { \hat A \hat F } ) }_{q_2}
\lesssim u^\eta.
\end{align}
\end{lemma}

\begin{proof}[Proof of Proposition~\ref{prop:f_frac} assuming Lemma~\ref{lemma:Uterms}]
By  hypothesis, $n_d + \eta < (d - 2 + \rho \wedge 2) \wedge ( \tfrac 1 2 d + 2 + \rho )$.
By the strict inequality,
we can pick $\sigma$ slightly smaller than $ \rho \wedge 2$ so that
we have $n_d + \eta < d - 2 + \sigma$ still.

Let $|\alpha| \le n_d$ and $u \in [0,1]$.
To bound $U_u \hat f_\alpha$, it suffices to bound the $L^1$ norm of $U_u$ applied to \eqref{eq:f_decomp}. We use \eqref{eq:X-Y} and bound the $L^1$ norm of, e.g.,
\begin{align}
\label{eq:1.111}
\( \prod_{n=1}^i \frac{ \hat A_{\delta_n} }{\hat A } (k+\uvec) \)
U_u( \frac{ \hat E_{\alpha_2} }{ \hat A \hat F } )(k)
\( \prod_{m=1}^j \frac{ \hat F_{\gamma_m}  }{ \hat F } (k - \uvec) \).
\end{align}
Recall from Lemma~\ref{lem:AFEbds} that,
under our current hypotheses,
\begin{align}
\frac{ \hat F_\gamma } { \hat F } ,
\frac{ \hat A_\gamma } { \hat A }
	& \in L^q
	\quad ( q\inv > \frac{ \abs \gamma  } d ), \qquad
\frac{ \hat E_{\alpha_2} }{\hat A \hat F}  \in L^q
	\quad  (q\inv > \frac{ 2 - \sigma + \abs{\alpha_2}  }d).
\end{align}
By H\"older's inequality, the $L^r$ norm of \eqref{eq:1.111} is bounded by
\begin{align}
\( \prod_{n=1}^i \Bignorm{ \frac{ \hat A_{\delta_n} }{\hat A } }_{r_n} \)
\Bignorm{ U_u( \frac{ \hat E_{\alpha_2} }{ \hat A \hat F } ) }_q
\( \prod_{m=1}^j \Bignorm{ \frac{ \hat F_{\gamma_m}  }{ \hat F } }_{q_m} \),
\end{align}
where $r$ can be any number that satisfies
\begin{align}
\frac 1 r &= \( \sum_{n=1}^i \frac 1 {r_n} \) + \frac 1 q + \( \sum_{m=1}^j \frac 1 {q_m} \) \nl
&> \frac{ \sum_{n=1}^i \abs{\delta_n}  } d  + \frac{ 2 - \sigma + \abs{\alpha_2} + \eta } d + \frac{ \sum_{m=1}^j \abs{\gamma_m}  } d
= \frac{  \abs \alpha + \eta + 2 - \sigma } d.
\end{align}
In the above, $r_n$ and $q_m$ are dictated by Lemma~\ref{lem:AFEbds}, and $q$ is
dictated by Lemma~\ref{lemma:Uterms}.
Since $\abs \alpha + \eta \le n_d + \eta < d - 2 + \sigma$, we can take $r=1$.
Also, by Lemma~\ref{lemma:Uterms}, the norm contains a factor $u^\eta$ from $\norm{ U_u( \frac{ \hat E_{\alpha_2} }{ \hat A \hat F } ) }_q$.
Similar terms with $U_u$ instead applied to one of the
$\hat F_\gamma /\hat F $ or $\hat A_\gamma /\hat A$ factors are treated in the same way,
using Lemma~\ref{lemma:Uterms}.
\end{proof}

Before we prove Lemma~\ref{lemma:Uterms}, we give a
further consequence of the good moment behaviour of $E$ in \eqref{eq:E_condition}.
The next lemma is similar to Lemma~\ref{lemma:Ek} and bounds $U_u \hat E$ and its derivatives.
Instead of the good powers of $\abs k$ in \eqref{eq:Ek}, we now have good powers of $\abs{ k \pm \uvec}$.

\begin{lemma} \label{lemma:UE}
Suppose $E: \Z^d \to \R$ is $\Z^d$-symmetric,
has vanishing zeroth and second moments as in
\eqref{eq:E_condition},
and satisfies $\abs{ E(x) } \le  K \abs x^{-(d+2 + \rho)}$ for some $K, \rho>0$.
Choose $\sigma \in (0, \rho)$ with $\sigma \le 2$, choose $0 \le \eta \le 1$,
and let $\alpha$ be a multi-index
with
$\abs \alpha +\eta < 2 + \sigma$.
Then there is a constant $c = c(\sigma, \rho, d)$ such that
\begin{align} \label{eq:UE_small}
\abs{  U_u \hat E_{\alpha}(k) }
\le cK \,  u^\eta \Big( \abs{ k+\uvec }^{2 + \sigma - \eta - \abs{\alpha} } + \abs{ k-\uvec }^{2 + \sigma - \eta - \abs{\alpha} } \Big).
\end{align}
\end{lemma}

\begin{proof}
We prove that
\begin{align}
\label{eq:UE_detail}
\abs{ U_u  \hat E_\alpha(k) } \le cK \cdot
\begin{cases}
u^{1+\sigma }\abs k + u \abs k^{1+\sigma} 	& (\abs \alpha = 0)
\\
u^{2+\sigma - \abs \alpha} + u \abs k^{1+\sigma - \abs \alpha} & (1 \le \abs \alpha < 1+\sigma)    \\
u^\sigma & (\sigma\le 1, \; \abs \alpha = 2)
\\
u^{\sigma-1} & (1<\sigma \le 2, \; \abs \alpha =3).
\end{cases}
\end{align}
The claim \eqref{eq:UE_small} then follows from the inequalities $\abs k, u \le 2\inv (\abs{ k + \uvec } + \abs { k - \uvec })$ and $ (\abs a + \abs b)^p \le 2^p( \abs a^p + \abs b^p)$ with $p = 2+\sigma-\eta - \abs \alpha \ge 0$.
For example, for the case $|\alpha|=0$ we use
\begin{equation}
    u^{1+\sigma}|k| = u^\eta u^{1+\sigma - \eta}|k|
    \lesssim
    u^\eta (|k+\uvec| + |k-\uvec|)^{1+\sigma - \eta + 1}
    \lesssim
    u^\eta (|k+\uvec|^{2+\sigma - \eta} + |k-\uvec|^{2+\sigma - \eta}).
\end{equation}

By definition,
\begin{align}
\label{eq:Esine}
\frac 1 {2i} U_u \hat E_\alpha(k)
= \sum_{x\in\Z^d}  \sin(ux_1)  (ix)^\alpha E(x) e^{ik\cdot x}
.
\end{align}
For the last two lines of \eqref{eq:UE_detail},
since $2 + \sigma - \abs \alpha \in [0,1]$ we can use
$\abs{ \sin(ux_1) } \lesssim  \abs{ux_1}^{ 2+ \sigma -\abs\alpha }$, together
with the fact that $2+\sigma < 2+\rho$ and the assumed decay of $E(x)$, to see that
\begin{align}
\abs{ U_u \hat E_\alpha (k) }
&\lesssim  \sum_{x\in\Z^d} \abs{ux_1}^{ 2+ \sigma -\abs\alpha } \abs x^{\abs\alpha} \abs{E(x)}
\lesssim Ku^{2+\sigma - \abs \alpha} ,
\end{align}
as claimed.

For the first two lines of \eqref{eq:UE_detail}, we instead make the decomposition
\begin{align}
\frac 1 {2i} U_u \hat E_\alpha(k)
= \sum_{x\in\Z^d} \big( \sin(ux_1) - ux_1 \big) (ix)^\alpha E(x) e^{ik\cdot x}
+ \sum_{x\in\Z^d} ux_1 (ix)^\alpha E(x) e^{ik\cdot x} .
\label{eq:5.14}
\end{align}
The second term is exactly $\frac u i  \hat E_{\alpha + e_1}(k)$, and
since $\abs \alpha + 1 < 2 + \sigma$, Lemma~\ref{lemma:Ek} implies that its
absolute value is bounded by a multiple of $K u \abs k^{2+\sigma - (\abs \alpha +1)}
=K u \abs k^{1+\sigma - \abs \alpha}$.
For the first term of \eqref{eq:5.14}, if $\abs \alpha = 0$  we write
$\tilde x = (x_2, \dots, x_d)$ and $\tilde k = (k_2, \dots, k_d)$
and use symmetry to obtain
\begin{align}
\Big\lvert \sum_{x\in\Z^d} \big( \sin(ux_1) - ux_1 \big) E(x) e^{ik\cdot x} \Big \rvert
&= \Big \lvert \sum_{x\in\Z^d} \big( \sin(ux_1) - ux_1 \big) E(x) e^{i\tilde k\cdot \tilde x} \sin(k_1x_1) \Big \rvert \nl
&\lesssim \sum_{x\in\Z^d} \abs{ ux_1 }^{1+\sigma} \abs {E(x) } \abs{ k_1 x_1 }
\lesssim K u^{1+\sigma} \abs k,
\end{align}
where we used $1+\sigma\le 3$ to bound the difference between the sine and
its linear part, as well as $2+ \sigma < 2+ \rho$ and the decay of $E(x)$.
Finally, if $1 \le \abs \alpha < 1+\sigma$, then $2 + \sigma - \abs \alpha \in[1, 3]$, and
\begin{align}
\Big\lvert \sum_{x\in\Z^d} \big( \sin(ux_1) - ux_1 \big) (ix)^\alpha E(x) e^{ik\cdot x} \Big \rvert
&\lesssim  \sum_{x\in\Z^d} \abs{ ux_1 }^{2+\sigma - \abs \alpha} \abs x^{\abs \alpha}\abs {E(x) }
\lesssim K u^{2+\sigma - \abs \alpha} ,
\end{align}
as desired.
This completes the proof.
\end{proof}

\subsubsection{Proof of Lemma~\ref{lemma:Uterms}}
\label{sec:Uterms}

It remains to prove Lemma~\ref{lemma:Uterms}.
When it is not possible to take a full derivative as
required by Lemma~\ref{lemma:h_diff}, we use the following lemma.
Its role in the proof of Lemma~\ref{lemma:Uterms} is analogous to the role
played by Lemma~\ref{lem:FTbdd} in the proof of
Lemma~\ref{lem:AFEbds}.

\begin{lemma}
\label{lem:UFTbdd}
Let $h: \Z^d \to \R$ obey $|h(x)| \le K\nnnorm x^{-b}$ for some $K,b>0$,
and let $0 \le \eta \le 1$.
\begin{enumerate}[label=(\roman*)]
\item
If $b  > d + \eta$ then
$\hat h \in L^\infty(\T^d)$ and $\norm{ U_u \hat h }_\infty \le c_{b,\eta} u^\eta K $.

\item
If $\half d + \eta < b \le d + \eta$ then
$\hat h \in L^q(\T^d)$
and $\norm{ U_u \hat h }_{q} \le c_{d,b,q} u^\eta K$  for all $1 \le q < d/(d-b + \eta)$.
\end{enumerate}
\end{lemma}

\begin{proof}
(i) The claim that $\hat h \in L^\infty(\T^d)$ follows from the summability of $\nnnorm x^{-b}$.
For the claim on $U_u \hat h$,
we first observe that $U_u \hat h$ is the Fourier transform
of $2i \sin(ux_1) h(x)$, by definition.
With $\abs{ \sin (ux_1) } \lesssim \abs {ux_1}^\eta \le u^\eta \abs x^\eta$ and the summability of $\nnnorm x^{\eta - b }$, this leads to
\begin{align}
\label{eq:UFT_small}
\norm{ U_u \hat h }_\infty
\le \norm{ 2i \sin(ux_1) h(x) }_1
\lesssim  u^\eta \norm{ \abs x^{ \eta} h(x) }_1
\lesssim u^\eta K.
\end{align}

\smallskip\noindent
(ii)
Note that $ \abs x^\eta \abs{ h(x) } \le K \nnnorm x^{\eta - b}$ is in $\ell^p(\Z^d)$ for $p > d/(b-\eta)$.
The condition $b > \half d + \eta$ ensures that we can take $p$ less than $2$.
Using the same bound on $\sin(ux_1)$ as in part (i), and
again using the
boundedness of the Fourier transform as a map from $\ell^p(\Z^d)$ to $L^q(\T^d)$ when $p \in [1,2]$ and $p \inv + q \inv = 1$, we get
\begin{align}
\label{eq:UFT_large}
\norm{ U_u \hat h }_q
\le \norm{ 2i \sin(ux_1) h(x) }_p
\lesssim  u^\eta \norm{ \abs x^{ \eta} h(x) }_p
\lesssim u^\eta K
\end{align}
for $2 \le q<\frac{d}{d-b+\eta}$.
The extension to $1 \le q<\frac{d}{d-b+\eta}$ follows from the monotonicity of the $L^q(\T^d)$ norm in $q$.
The claim on $\hat h$ is similar and does not use $\sin(ux_1)$.
\end{proof}

\begin{proof}[Proof of Lemma~\ref{lemma:Uterms}]
Let $0 \le \eta \le 1$, $|\gamma|+\eta < (d-2+\rho \wedge 2) \wedge
(\frac 12 d + 2+\rho)$, $\sigma \in (0,\rho)$, and $\sigma \le 2$; these conditions
hold throughout the proof.
We start by proving \eqref{eq:Uterms} for $\hat F$. The proof for $\hat A$ is analogous so we omit it.

\medskip\noindent \emph{Bound on $U_u (\hat F_\gamma / \hat F) $.}
The statement is trivial for $\abs \gamma =0$, as $U_u(1) = 0$.
We treat the cases $\abs \gamma =1$ and $\abs \gamma \ge 2$ separately.

\smallskip \noindent \emph{Case 1: $\abs { \gamma } = 1$.}
Let $i, j  \in \{ 1, \dots, d \}$ and write $\gamma = e_i$.
We use Lemma~\ref{lemma:h_diff}.
It follows from the quotient rule, the infrared bound, $\grad_{ij} \hat F \in L^\infty$ and $\abs{ \grad_i \hat F (k) }, \abs{ \grad_j \hat F (k) } \lesssim \abs k$
(which follow from $|x|^2 F(x) \in \ell^1$), that
\begin{align}
\Bigabs{ \grad_j ( \frac{ \hat F_\gamma } { \hat F } ) }
= \Bigabs{ \frac{ \grad_{ij} \hat F } { \hat F } - \frac{ \grad_i \hat F \grad_j \hat F } { \hat F^2 } }
\lesssim \frac 1 { \abs k^{2} },
\end{align}
which is in $L^p$ for $p\inv > 2/d$.
Since $d>2$, $p=1$ is permitted.
This implies that $\hat F_\gamma / \hat F$ is weakly differentiable
and $\grad_j ( \hat F_\gamma / \hat F ) \in L^p$ for the same range of $p$.
Lemma~\ref{lemma:h_diff} then gives
$\norm{ U_u (  \hat F_\gamma / \hat F  ) }_q	\lesssim u^\eta$ for \begin{align} \label{eq:F:q-p_eta}
\frac 1 q = \frac 1 {p_\eta} = \frac 1 p - \frac{ 1-\eta } d > \frac{ 1 + \eta } d,
\end{align}
which is the desired result.

\smallskip \noindent \emph{Case 2: $\abs \gamma \ge 2$.}
By \eqref{eq:X-Y}, \begin{align} \label{eq:UFF_decomp}
U_u ( \frac{ \hat F _ \gamma } { \hat F }  )(k)
= (U_u \hat F_\gamma)(k) \hat F\inv(k+\uvec) + \hat F_\gamma (k - \uvec) U_u \hat F\inv (k).
\end{align}
We show each of the two terms on the right-hand side has $L^q$ norm
at most $u^\eta$ when $q\inv > (\abs \gamma + \eta)/d$.

For the first term of \eqref{eq:UFF_decomp},
we first estimate $U_u \hat F_\gamma$ by applying Lemma~\ref{lem:UFTbdd}
with $h(x) = x^\gamma F(x)$ and
\begin{align} \label{eq:UFF_pf}
b = - \abs \gamma  + d + 2 + \rho .
\end{align}
If $\abs \gamma + \eta < 2 + \rho$ then $b > d + \eta$, and Lemma~\ref{lem:UFTbdd}(i) gives $\norm{ U_u \hat F_\gamma }_\infty \lesssim u^\eta$.
On the other hand, if $2 + \rho \le \abs \gamma + \eta < \half d + 2 + \rho$ then $b \in ( \half d + \eta, d + \eta ]$, and Lemma~\ref{lem:UFTbdd}(ii) gives $\norm{ U_u \hat F_\gamma }_r \lesssim u^\eta$ for $r\inv > (\abs \gamma + \eta - 2 - \rho) / d$.
Since $\rho > 0$, we can combine the two subcases to get
\begin{align}
\norm{ U_u \hat F_{\gamma} }_{\frac d {\abs{\gamma}+\eta - 2} } \lesssim u^\eta
	\qquad ( 2 \le \abs \gamma < \half d + 2 + \rho - \eta ).
\end{align}
Then, by the infrared bound, it follows from H\"older's inequality that
\begin{equation}
\Bignorm{  (U_u \hat F_\gamma)(k) \hat F\inv(k+\uvec)  }_{q}
\lesssim  \norm{ U_u \hat F_{\gamma}  }_{\frac{d}{\abs \gamma + \eta - 2}} \Bignorm{ \frac{1}{|k+ \uvec|^2} }_p
\lesssim u^\eta ,
\end{equation}
provided that $q^{-1} = p^{-1}+  \frac{|\gamma| + \eta -2}{d} > \frac 2d +\frac{|\gamma| + \eta -2}{d}
= \frac{|\gamma| + \eta}{d}$, as desired.

For the second term of \eqref{eq:UFF_decomp},
we recall $\hat F_\gamma \in L^{ d / (\abs \gamma - 2) }$ from \eqref{eq:F_gamma}. To calculate $U_u \hat F\inv$, we use the symmetry of $F(x)$ to write \begin{align}
U_u \hat F(k)
&= 2i \sum_{x\in \Z^d} \sin(ux_1) F(x) e^{ik\cdot x}
= 2i^2 \sum_{x\in \Z^d} \sin(ux_1) F(x) e^{i\tilde k\cdot \tilde x} \sin(k_1x_1),
\end{align}
where $\tilde x = (x_2, \dots, x_d)$ and $\tilde k = (k_2, \dots, k_d)$.
Since $\abs{ \sin t } \le \abs t$,
\begin{align}
\abs{ U_u\hat F(k) }
&\le 2 \sum_{x\in \Z^d} u  \abs {x_1} \abs{ F(x) } \abs{ k_1 x_1}
\lesssim u \abs k.
\end{align}
By the inequalities $\abs k, u \le 2\inv (\abs{ k + \uvec } + \abs { k - \uvec })$ and $ (\abs a + \abs b)^p \le 2^p( \abs a^p + \abs b^p)$ with $p = 2-\eta\ge 0$,
\begin{align} \label{eq:UF0}
\abs{ U_u\hat F(k) }
\lesssim u \abs k
\le u^\eta   ( \abs{ k+\uvec }^{2-\eta} + \abs { k- \uvec }^{2-\eta} ),
\end{align}
and hence
\begin{align}
\label{eq:UF_inv}
\abs{ (U_u \hat F \inv )(k) }
&= \left \lvert \frac{-U_u \hat F(k) }{ \hat F(k+\uvec) \hat F(k-\uvec)  } \right \rvert
\lesssim u^\eta  \(  \frac 1 { \abs{k + \uvec}^{\eta} \abs{k - \uvec}^2 }
+  \frac 1 { \abs{k + \uvec}^2 \abs{k - \uvec}^{\eta} } \).
\end{align}
Then, it follows from H\"older's inequality that
\begin{align}
\Bignorm{ \hat F_\gamma (k - \uvec) U_u \hat F\inv (k) }_q
\lesssim u^\eta \norm{ \hat F_ \gamma }_{ \frac{d}{\abs \gamma - 2} }
	\Bignorm{ \frac1 { \abs { k \pm \uvec  }^ \eta } }_{r_1}
	\Bignorm{ \frac1 { \abs { k \pm \uvec  }^ 2 } }_{r_2} ,
\end{align}
provided that $q \inv = \frac { \abs \gamma -2 } d + r_1 \inv + r_2 \inv
>  \frac { \abs \gamma -2 } d + \frac \eta d + \frac 2 d
= \frac{ \abs \gamma + \eta } d$, as desired.

\medskip\noindent \emph{Bound on $U_u(\hat E_\gamma / (\hat A \hat F))$.}
We consider separately the two cases $\abs \gamma < \sigma$ and $\abs \gamma \ge \sigma$.

\smallskip \noindent \emph{Case 1: $\abs { \gamma } < \sigma$.}
By the quotient rule, \begin{align} \label{eq:DEAF}
\grad_j( \frac{ \hat E_{\gamma} }{ \hat A \hat F} )
= \frac{ \hat E_{\gamma + e_j } }{ \hat A \hat F} - \frac{ \hat E_{\gamma} }{ \hat A \hat F} \frac{ \hat A_{e_j} } {\hat A} - \frac{ \hat E_{\gamma} }{ \hat A \hat F} \frac{ \hat F_{e_j} } {\hat F},
\end{align}
which is in $L^p$ when $p\inv > ( 3 - \sigma + \abs {\gamma}  )/d$ by Lemma~\ref{lem:AFEbds} and H\"older's inequality. Since $\abs \gamma < \sigma$ and $d\ge 3$, $p=1$ is permitted.
This implies that $\hat E_\gamma / (\hat A \hat F)$ is weakly differentiable
and $\grad_j ( \hat E_\gamma / (\hat A \hat F) ) \in L^p$ for the same range of $p$.
Lemma~\ref{lemma:h_diff} then gives $\norm{ U_u (\hat E_{ \gamma } / (\hat A \hat F) ) }_q \lesssim u^\eta$ for \begin{align}
\frac 1 q = \frac 1 {p_\eta} = \frac 1 p - \frac{ 1-\eta } d > \frac{ 2 - \sigma + \abs {\gamma} +\eta} d,
\end{align}
which is the desired result.

\smallskip \noindent \emph{Case 2: $\abs { \gamma } \ge \sigma$.}
By \eqref{eq:X-Y},
\begin{align} \label{eq:UEAF_decomp}
U_u (\frac{ \hat E_{\gamma} }{ \hat A \hat F})(k) 	
&= \frac{ \hat E_{\gamma} }{ \hat A }(k + \uvec) U_u \hat F \inv (k)
	+ \frac { U_u \hat E_{\gamma }(k)  } { \hat A (k + \uvec) \hat F (k - \uvec)}
	+ U_u \hat A \inv (k) \frac{ \hat E_{\gamma} }{ \hat F }(k - \uvec).
\end{align}
We show all three terms on the right-hand side have $L^q$ norm $u^\eta$ when $q\inv > (2 - \sigma + \abs \gamma+ \eta)/d$. For the first term, we have from \eqref{eq:UF_inv} that
$\norm{ U_u \hat F \inv }_r \lesssim u^\eta$ when $r\inv > (2+\eta)/d$.
And, since $\abs \gamma \ge \sigma$, from \eqref{eq:EAEF} we have $\hat E_{\gamma} / \hat A \in L^p$ for $p\inv > ( \abs \gamma - \sigma )/d$.
H\"older's inequality then gives $\norm{ \frac{ \hat E_{\gamma} }{ \hat A }(k + \uvec) U_u \hat F \inv (k)  }_q \lesssim u^\eta $ for $q \inv = p \inv + r \inv > (2 - \sigma + \abs \gamma + \eta)/d$, as desired.
The last term of \eqref{eq:UEAF_decomp} is analogous.

For the  middle term of \eqref{eq:UEAF_decomp}, we further divide into two subcases
according to how $\abs \gamma + \eta$ compares with $2 + \sigma$.
If $2 + \sigma \le \abs \gamma + \eta < \half d + 2 +\rho$,
an application of Lemma~\ref{lem:UFTbdd}
with $h(x) = x^\gamma E(x)$ and
\begin{align}
b = - \abs \gamma + d + 2 + \rho
\end{align}
gives, as in the paragraph containing \eqref{eq:UFF_pf},
$\norm{ U_u \hat E_\gamma }_\infty \lesssim u^\eta$ when $\abs \gamma + \eta < 2 + \rho$,
and $\norm{ U_u \hat E_\gamma }_r \lesssim u^\eta $ for $r \inv > (\abs \gamma + \eta - 2 - \rho)/ d$ when $2 + \rho \le \abs \gamma + \eta < \half d + 2 + \rho$.
We use $\sigma < \rho$ to combine the two as
\begin{align} \label{eq:UE}
\norm{ U_u \hat E_{\gamma} }_r \lesssim u^\eta 	\qquad ( r \inv > \frac{ \abs \gamma + \eta - 2 - \sigma} d ).
\end{align}
The infrared bounds and H\"older's inequality then yield
\begin{align}
\Bignorm{ \frac { U_u \hat E_{\gamma }(k)  } { \hat A (k + \uvec) \hat F (k - \uvec)}  }_q
\lesssim \norm{ U_u \hat E_{\gamma } }_r
	\Bignorm{ \frac1 { \abs { k + \uvec  }^ 2 } }_{p_1}
	\Bignorm{ \frac1 { \abs { k + \uvec  }^ 2 } }_{p_2}
\lesssim u^\eta ,
\end{align}
provided that $q\inv = r\inv + p_1 \inv + p_2 \inv > \frac{ \abs \gamma + \eta - 2 - \sigma} d + \frac 2 d + \frac 2 d
= \frac {2 - \sigma + \abs \gamma+ \eta } d$.
If instead $\abs \gamma + \eta < 2 + \sigma$ then we apply Lemma~\ref{lemma:UE} to obtain
\begin{align}
\abs{  U_u\hat E_{\gamma}(k) }
\lesssim  u^\eta ( \abs{ k+\uvec }^{2 + \sigma - \eta - \abs{\gamma} } + \abs{ k-\uvec }^{2 + \sigma - \eta - \abs{\gamma} } ).
\end{align}
With the infrared bounds, this gives
\begin{align}
\Bigabs { \frac { U_u \hat E_{\gamma }(k)  } { \hat A (k + \uvec) \hat F (k - \uvec)}  }
\lesssim
u^\eta \( \frac 1 { \abs{k + \uvec }^{\abs \gamma - \sigma + \eta} \abs{ k - \uvec}^2 } + \frac 1 { \abs{k + \uvec }^{2} \abs{ k - \uvec}^{\abs \gamma - \sigma + \eta} } \).
\end{align}
Since $\abs \gamma \ge \sigma$, splitting $\abs{ k \pm \uvec}$ with
H\"older's inequality directly gives $\norm{ \frac { U_u \hat E_{\gamma }(k)  } { \hat A (k + \uvec) \hat F (k - \uvec)}  }_q \lesssim u^\eta$ for $q \inv > (2 - \sigma + \abs \gamma + \eta)/d$.

This completes the proof of Lemma~\ref{lemma:Uterms}.
\end{proof}

\begin{appendix}

\section{Weak derivatives}
\label{appendix:weak}

We summarise here the properties of weak derivatives that we need.
More can be found in \cite[Chapter~5]{Evan10}, \cite[Chapter~4]{EG15}, or \cite[Chapter~1]{Mazy11}.
We restrict our discussion to functions on the torus $\T^d$, which simplifies matters
because $\T^d$ is compact and has no boundary.\footnote{The weak derivative can be defined for functions on an open subset $\Omega \subset \R^d$ that are merely locally integrable,
using the space of infinitely differentiable \emph{compactly supported} functions as test functions.}
Let $C^\infty (\T^d)$ denote the space of infinitely differentiable \emph{test functions}
$\phi: \T^d\to \R$.
Since the torus is compact, these test functions have compact support.
 All integrals in this appendix are with respect to the
normalised Lebesgue measure $dk/(2\pi)^d$ on $\T^d$.

\begin{definition}
Let $n$ be a positive integer and let $u \in L^1(\T^d)$.
We say that $u$ is $n$ times weakly differentiable if for
each multi-index $\alpha$ with $|\alpha| \le n$ there is a function $v_\alpha \in
 L^1(\T^d)$ such that for all test functions $\phi \in C^\infty (\T^d)$,
\begin{align} \label{eq:IBP}
\int_{\T^d} u \grad^\alpha \phi= (-1)^{\abs \alpha} \int_{\T^d} v_\alpha \phi.
\end{align}
In this case, we call $v_\alpha$ the $\alpha^{\rm th}$ weak derivative of $u$ and
write $\grad^\alpha u = v_\alpha$.
\end{definition}

If $u$ is classically differentiable then \eqref{eq:IBP} holds by
the usual integration by parts formula,
so weak derivatives are consistent with classical derivatives. The weak derivative is unique and satisfies standard calculus rules; see, \eg, \cite[Sec. 5.2.3, Theorem~1]{Evan10} or \cite[Theorem~4.4]{EG15}.
We make heavy use of the product and quotient rules, in a form that slightly generalises
\cite[Theorem~4.4]{EG15}, so we give the proofs here.

\begin{lemma}[Product rule] \label{lemma:product_rule}
Let $j \in \{ 1, \dots, d \}$. Suppose
that $f, g$ are weakly differentiable, and that $f, \grad_j f \in L^p$ and $g, \grad_j g \in L^q$ with
$\frac 1 p + \frac 1 q \le 1$, for all $j$.
Then $fg$ is weakly differentiable, with
\begin{align}
\grad_j (fg) = (\grad_j f) g + f (\grad_j g).
\end{align}
\end{lemma}

\begin{proof}
By hypothesis and H\"older's inequality, all of $fg$, $(\grad_j f)g$, and $f(\grad_j g)$ are in $L^1$.
By convolving with a mollifier as in \cite[Section~4.2.1]{EG15}, we take smooth approximations
$f^\eps, g^\eps \in C^\infty$ such that $ f^\eps \to f $ and $\grad_j f^\eps \to \grad_j f$ in $L^p$, and $ g^\eps \to g $ and $\grad_j g^\eps \to \grad_j g$ in $L^q$.
Then, for $\phi\in C^\infty(\T^d)$,
\begin{align}
\int_{\T^d} fg \grad_j \phi
= \lim_{\eps \to 0} \int_{\T^d} f^\eps g^\eps \grad_j \phi
&= - \lim_{\eps \to 0} \int_{\T^d} [(\grad_j f^\eps) g^\eps + f^\eps (\grad_j g^\eps) ] \phi
= -  \int_{\T^d} [(\grad_j f) g + f (\grad_j g) ] \phi,
\end{align}
where the second equality uses classical integration by parts. This verifies the definition of the weak derivative in \eqref{eq:IBP} and completes the proof.
\end{proof}

For the proof of the quotient rule, we use the ACL (absolutely continuous on lines) characterisation of weak derivative (see \cite[Section~1.1.3]{Mazy11} or \cite[Theorem~4.21]{EG15}), which is as follows.
A function $u\in L^1(\T^d)$ is weakly differentiable if and only if: (i)
the restriction of $u$ to almost every line in the coordinate directions is absolutely continuous and (ii) the a.e.-defined classical derivatives are integrable over $\T^d$.
In this case, the weak derivatives coincide with the classical derivatives.

\begin{lemma}[Quotient rule] \label{lemma:quotient_rule} 
Let $j \in \{ 1, \dots, d \}$.
Suppose that $f, g$ are weakly differentiable, that $g$ has countably many zeros, that $f/g \in L^1$, and that $[ (\grad_j f)g - f (\grad_j g)]/g^2 \in L^1$ for all $j$.
Then $f/g$ is weakly differentiable, with
\begin{align} \label{eq:quotient}
\grad_j \( \frac f g \) = \frac{ (\grad_j f)g - f (\grad_j g) }{g^2}.
\end{align}
\end{lemma}

\begin{proof}
Since $f$ and $g$ are weakly differentiable, we know from the ACL characterisation that
$f$ and $g$ are absolutely continuous on almost every line in the coordinate directions. Since $g$ has countably many zeros only, we can assume that $g$ is nonzero on these lines.
Moreover, since $\T^d$ is compact, $g$ is in fact bounded away from $0$
on these lines, so that $f/g$ is absolutely continuous on the lines.
By the classical quotient rule, $f/g$ has classical partial derivatives almost everywhere, given by the right-hand side of \eqref{eq:quotient}, which are integrable by assumption.
This verifies the ACL characterisation for $f/g$ and concludes the proof.
\end{proof}

The next lemma shows that weak derivatives preserve the classical fact that
multiplication by a power corresponds to differentiation of the Fourier transform,
a fact we rely on heavily.
We write $\Fcal [f]   =\hat f$ for the $L^2$ Fourier transform of $f \in \ell^2(\Z^d)$.

\begin{lemma} \label{lemma:weak_Fourier}
Let $f: \Z^d \to \R$ and let $n$ be a positive integer.
Suppose that $\abs{x}^n f(x) \in \ell^2(\Z^d)$.
Then $\hat f$ is $n$ times weakly differentiable, and
the $\alpha^{\rm th}$ weak derivative of $\hat f$ is given by
\begin{align}
\grad^\alpha \hat f = \Fcal [  (ix)^\alpha f(x)  ]
	\qquad (\abs \alpha \le n) .
\end{align}
\end{lemma}

\begin{proof}
Let $\abs \alpha \le n$ and
$\phi\in C^\infty(\T^d)$.
Since $x^\alpha f(x) \in \ell^2(\Z^d)$,
by Parseval's relation and integration by parts we have
\begin{align}
\int_{\T^d} \Fcal [ (ix)^\alpha f(x) ] \bar \phi
&= \sum_{x\in\Z^d} (ix)^\alpha f(x)  \overline{ \Fcal\inv[ \phi ] }	\nl
&= (-1)^{\abs \alpha} \sum_{x\in\Z^d}  f(x)  \overline{ (ix)^\alpha \Fcal\inv[ \phi ] } \nl
&= (-1)^{\abs \alpha} \sum_{x\in\Z^d}  f(x)  \overline{ \Fcal\inv[ \grad^\alpha \phi ] }
= (-1)^{\abs \alpha} \int_{\T^d} \hat f \, \overline{\grad^\alpha \phi}.
\end{align}
This verifies the definition of weak derivative~\eqref{eq:IBP}, because $\phi$ is real.
\end{proof}

\end{appendix}

\section*{Acknowledgements}
The work of both authors was supported in part by NSERC of Canada.
We thank Takashi Hara and Emmanuel Michta for comments on a preliminary version.
YL thanks Fanze Kong for discussions about weak derivatives.


\begin{thebibliography}{10}

\bibitem{Aize97}
M.~Aizenman.
\newblock On the number of incipient spanning clusters.
\newblock {\em Nucl. Phys. B [FS]}, {\bf 485}:551--582, (1997).

\bibitem{AN84}
M.~Aizenman and C.M. Newman.
\newblock Tree graph inequalities and critical behavior in percolation models.
\newblock {\em J. Stat. Phys.}, {\bf 36}:107--143, (1984).

\bibitem{BO13}
\'{A}. B\'{e}nyi and T.~Oh.
\newblock The {Sobolev} inequality on the torus revisited.
\newblock {\em Publ. Math. Debrecen}, {\bf 83}:359--374, (2013).

\bibitem{BHK18}
E.~Bolthausen, R.~van~der Hofstad, and G.~Kozma.
\newblock Lace expansion for dummies.
\newblock {\em Ann.\ I.\ Henri Poincar\'{e} Probab.\ Statist.}, {\bf
  54}:141--153, (2018).

\bibitem{BHH21}
D.C. Brydges, T.~Helmuth, and M.~Holmes.
\newblock The continuous-time lace expansion.
\newblock {\em Commun. Pure Appl. Math.}, {\bf 74}:2251--2309, (2021).

\bibitem{BS85}
D.C. Brydges and T.~Spencer.
\newblock Self-avoiding walk in 5 or more dimensions.
\newblock {\em Commun. Math. Phys.}, {\bf 97}:125--148, (1985).

\bibitem{CH20}
S.~Chatterjee and J.~Hanson.
\newblock Restricted percolation critical exponents in high dimensions.
\newblock {\em Commun. Pure Appl. Math.}, {\bf 73}:2370--2429, (2020).

\bibitem{CHS23}
S.~Chatterjee, J.~Hanson, and P.~Sosoe.
\newblock Subcritical connectivity and some exact tail exponents in high
  dimensional percolation.
\newblock {\em Commun. Math. Phys.}, {\bf 403}:83--153, (2023).

\bibitem{CS09a}
L.-C. Chen and A.~Sakai.
\newblock Critical behavior and the limit distribution for long-range oriented
  percolation. {II}: {Spatial} correlation.
\newblock {\em Probab. Theory Related Fields}, {\bf 145}:435--458, (2009).

\bibitem{DPV12}
E.~Di~Nezza, G.~Palatucci, and E.~Valdinoci.
\newblock Hitchhiker's guide to the fractional {Sobolev} spaces.
\newblock {\em Bull. Sci. Math.}, {\bf 136}:521--573, (2012).

\bibitem{Evan10}
L.C. Evans.
\newblock {\em Partial Differential Equations}.
\newblock American Mathematical Society, Providence, 2nd edition, (2010).

\bibitem{EG15}
L.C. Evans and R.F. Gariepy.
\newblock {\em Measure Theory and Fine Properties of Functions, Revised
  Edition}.
\newblock Chapman and Hall/CRC, New York, (2015).

\bibitem{FH17}
R.~Fitzner and R.~van~der Hofstad.
\newblock Mean-field behavior for nearest-neighbor percolation in $d>10$.
\newblock {\em Electron. J. Probab.}, {\bf 22}:1--65, (2017).

\bibitem{FH21}
R.~Fitzner and R.~van~der Hofstad.
\newblock {N}o{B}{L}{E} for lattice trees and lattice animals.
\newblock {\em J. Stat. Phys.}, {\bf 185}:paper 13, (2021).

\bibitem{Foll15}
G.B. Folland.
\newblock {\em A Course in Abstract Harmonic Analysis}.
\newblock CRC Press, New York, 2nd edition, (2015).

\bibitem{Graf14}
L.~Grafakos.
\newblock {\em Classical Fourier Analysis}.
\newblock Springer, New York, 3rd edition, (2014).

\bibitem{Hara08}
T.~Hara.
\newblock Decay of correlations in nearest-neighbor self-avoiding walk,
  percolation, lattice trees and animals.
\newblock {\em Ann. Probab.}, {\bf 36}:530--593, (2008).

\bibitem{HHS03}
T.~Hara, R.~van~der Hofstad, and G.~Slade.
\newblock Critical two-point functions and the lace expansion for spread-out
  high-dimensional percolation and related models.
\newblock {\em Ann. Probab.}, {\bf 31}:349--408, (2003).

\bibitem{HS90a}
T.~Hara and G.~Slade.
\newblock Mean-field critical behaviour for percolation in high dimensions.
\newblock {\em Commun. Math. Phys.}, {\bf 128}:333--391, (1990).

\bibitem{HS90b}
T.~Hara and G.~Slade.
\newblock On the upper critical dimension of lattice trees and lattice animals.
\newblock {\em J. Stat. Phys.}, {\bf 59}:1469--1510, (1990).

\bibitem{HS92a}
T.~Hara and G.~Slade.
\newblock Self-avoiding walk in five or more dimensions. {I.} {The} critical
  behaviour.
\newblock {\em Commun.\ Math.\ Phys.}, {\bf 147}:101--136, (1992).

\bibitem{HH17book}
M.~Heydenreich and R.~van~der Hofstad.
\newblock {\em Progress in High-Dimensional Percolation and Random Graphs}.
\newblock Springer International Publishing Switzerland, (2017).

\bibitem{HJ04}
R.~van~der Hofstad and A.A. J\'arai.
\newblock The incipient infinite cluster for high-dimensional unoriented
  percolation.
\newblock {\em J. Stat. Phys.}, {\bf 114}:625--663, (2004).

\bibitem{HMS23}
T.~Hutchcroft, E.~Michta, and G.~Slade.
\newblock High-dimensional near-critical percolation and the torus plateau.
\newblock {\em Ann. Probab.}, {\bf 51}:580--625, (2023).

\bibitem{KN09}
G.~Kozma and A.~Nachmias.
\newblock The {Alexander}--{Orbach} conjecture holds in high dimensions.
\newblock {\em Invent. Math.}, {\bf 178}:635--654, (2009).

\bibitem{KN11}
G.~Kozma and A.~Nachmias.
\newblock Arm exponents in high dimensional percolation.
\newblock {\em J. Am. Math. Soc.}, {\bf 24}:375--409, (2011).

\bibitem{LL10}
G.F. Lawler and V.~Limic.
\newblock {\em Random Walk: A Modern Introduction}.
\newblock Cambridge University Press, Cambridge, (2010).

\bibitem{LS24b}
Y.~Liu and G.~Slade.
\newblock Gaussian deconvolution and the lace expansion for spread-out models.
\newblock To appear in {\it Ann.\ Inst.\ H.\ Poincar\'e Probab.\ Statist.}
  Preprint, \url{https://arxiv.org/pdf/2310.07640.pdf}, (2023).

\bibitem{MS93}
N.~Madras and G.~Slade.
\newblock {\em The Self-Avoiding Walk}.
\newblock Birkh{\"a}user, Boston, (1993).

\bibitem{Mazy11}
V.~Maz'ya.
\newblock {\em Sobolev Spaces: with Applications to Elliptic Partial
  Differential Equations}.
\newblock Springer, Heidelberg, 2nd edition, (2011).

\bibitem{MS22}
E.~Michta and G.~Slade.
\newblock Asymptotic behaviour of the lattice {Green} function.
\newblock {\em ALEA, Lat. Am. J. Probab. Math. Stat.}, {\bf 19}:957--981,
  (2022).

\bibitem{Saka07}
A.~Sakai.
\newblock Lace expansion for the {Ising} model.
\newblock {\em Commun. Math. Phys.}, {\bf 272}:283--344, (2007).
\newblock Correction: A.~Sakai. Correct bounds on the Ising lace-expansion
  coefficients. {\it Commun. Math. Phys.}, {\bf 392}:783--823, (2022).

\bibitem{Saka15}
A.~Sakai.
\newblock Application of the lace expansion to the $\varphi^4$ model.
\newblock {\em Commun. Math. Phys.}, {\bf 336}:619--648, (2015).

\bibitem{Slad20_Kotani}
G.~Slade.
\newblock Kotani's theorem for the {Fourier} transform.
\newblock Unpublished note. \url{https://arxiv.org/pdf/2006.06532}, (2020).

\bibitem{Slad22_lace}
G.~Slade.
\newblock A simple convergence proof for the lace expansion.
\newblock {\em Ann.\ I.\ Henri Poincar\'{e} Probab.\ Statist.}, {\bf
  58}:26--33, (2022).

\bibitem{Uchi98}
K.~Uchiyama.
\newblock Green's functions for random walks on $\mathbb{Z}^{N}$.
\newblock {\em Proc.\ London Math.\ Soc.}, {\bf 77}:215--240, (1998).

\end{thebibliography}

\end{document}